\documentclass[a4paper,11pt]{article}

\usepackage[margin=3.2cm]{geometry}

\usepackage{graphics}

\usepackage{docmute}

\usepackage[utf8]{inputenc}

\usepackage{mathtools}

\usepackage{amsmath}

\usepackage{amsfonts}

\usepackage{amssymb}

\usepackage{amsthm}
\usepackage{aliascnt}

\usepackage{makeidx}
\makeindex

\usepackage{textcomp}
\usepackage[sb]{libertine}
\usepackage[varqu,varl]{zi4}%
\usepackage[libertine,bigdelims,vvarbb]{newtxmath}
\usepackage[supscaled=1.2,raised=-.13em]{superiors}
\usepackage{bm}
\useosf

\usepackage[breaklinks,linktocpage]{hyperref}
\usepackage[hyperpageref]{backref}

\usepackage{url}
\usepackage{breakurl}

\usepackage{tikz-cd}

\usepackage{enumitem}
\usepackage[UKenglish]{babel}
\usepackage[makeroom]{cancel}
\usepackage{pstricks-add}

\definecolor{shadecolor}{rgb}{0.88,0.91,0.95}       
\usepackage{tcolorbox}

\usepackage{dynkin-diagrams}
\usetikzlibrary{backgrounds}

\usepackage{booktabs}

\usepackage{fancyhdr}
\newcommand{\Z}{\mathbb{Z}}
\newcommand{\N}{\mathbb{N}}
\newcommand{\R}{\mathbb{R}}
\newcommand{\C}{\mathbb{C}}

\newcommand{\Ker}{\operatorname{Ker}}

\newcommand{\spann}{\operatorname{span}}

\renewcommand{\Im}{\operatorname{Im}}

\newcommand{\Hol}{\operatorname{Hol}}

\newcommand{\Rm}{\operatorname{Rm}}

\newcommand{\GL}{\operatorname{GL}}

\newcommand{\SO}{\operatorname{SO}}

\newcommand{\so}{\mathfrak{so}}

\newcommand{\Sp}{\operatorname{Sp}}
\newcommand{\Spin}{\operatorname{Spin}}

\newcommand{\CP}{\mathbb{CP}}

\newcommand{\Spec}{\operatorname{Spec}}

\newcommand{\Id}{\operatorname{Id}}

\renewcommand{\|}[1]{\left| \left| #1 \right| \right|}

\newcommand{\<}{\langle}
\renewcommand{\>}{\rangle}

\newcommand{\proj}{\operatorname{proj}}

\newcommand{\vol}{\operatorname{vol}}

\newcommand{\inj}{\operatorname{inj}}

\newcommand{\supp}{\operatorname{supp}}

\newcommand{\ind}{\operatorname{ind}}

\renewcommand{\d}{\operatorname{d} \!}

\renewcommand{\tilde}[1]{\widetilde{#1}}

\newcommand{\XEH}{X_{\text{EH}}}

\newcommand\Item[1][]{%
  \ifx\relax#1\relax  \item \else \item[#1] \fi
  \abovedisplayskip=0pt\abovedisplayshortskip=0pt~\vspace*{-\baselineskip}}

\author{Daniel Platt}
\date{\today}
\title{Existence of torsion-free $G_2$-structures on resolutions of $G_2$-orbifolds using weighted Hölder norms}

\usepackage[capitalise]{cleveref}

\numberwithin{equation}{section}

\newtheorem{proposition}[equation]{Proposition}
\crefname{proposition}{Proposition}{Propositions}

\newtheorem{lemma}[equation]{Lemma}
\crefname{lemma}{Lemma}{Lemmas}

\newtheorem{corollary}[equation]{Corollary}
\crefname{corollary}{Corollary}{Corollaries}

\newtheorem*{corollary*}{Corollary}
\crefname{corollary*}{Corollary}{Corollaries}

\newtheorem{theorem}[equation]{Theorem}
\crefname{theorem}{Theorem}{Theorems}

\newtheorem*{theorem*}{Theorem}
\crefname{theorem*}{Theorem}{Theorems}

\crefname{claim}{Claim}{Claims}

\theoremstyle{remark}

\crefname{question}{Question}{Questions}

\newtheorem{definition}[equation]{Definition}
\crefname{definition}{Definition}{Definitions}

\crefname{example}{Example}{Examples}

\newtheorem{remark}[equation]{Remark}
\crefname{remark}{Remark}{Remarks}

\crefname{assumption}{Assumption}{Assumptions}

\begin{document}

\maketitle

\begin{abstract}
 The resolution of the $G_2$-orbifold $T^7/\Gamma$, where $\Gamma$ is a suitably chosen finite group, admits a $1$-parameter family of $G_2$-structures with small torsion $\varphi^t$, obtained by gluing in Eguchi-Hanson spaces.
It was shown in \cite{Joyce1996} that $\varphi^t$ can be perturbed to torsion-free $G_2$-structures $\tilde{\varphi}^t$ for small values of $t$.
Using norms adapted to the geometry of the manifold we give an alternative proof of the existence of $\tilde{\varphi}^t$.
This alternative proof produces the estimate $\|{ \tilde{\varphi}^t-\varphi^t}_{C^0} \leq ct^{5/2}$.
This is an improvement over the previously known estimate $\|{ \tilde{\varphi}^t-\varphi^t}_{C^0} \leq ct^{1/2}$.
As part of the proof, we show that Eguchi-Hanson space admits a unique (up to scaling) harmonic form with decay, which is a result of independent interest.
\end{abstract}

\tableofcontents

\setlength{\parskip}{0.3cm}

\pagebreak

\section{Introduction}

The first compact examples of Riemannian manifolds with holonomy equal to $G_2$ were constructed in \cite{Joyce1996} by resolving an orbifold of the form $T^7/\Gamma$, where $\Gamma$ is a finite group of isometries of $T^7$.
This was done by constructing $G_2$-structures with small torsion, and subsequently perturbing them to torsion-free $G_2$-structures.
This perturbation made use of a general existence result for torsion-free $G_2$-structures that holds on all $7$-manifolds.
An immediate question is:
how far away is the torsion-free $G_2$-structure from the $G_2$-structure with small torsion?
This is important in applications, such as the construction of associative submanifolds and $G_2$-instantons, see \cite{Dwivedi2023,Platt2024}.

The construction from \cite{Joyce1996} is a generalisation of the \emph{Kummer construction} for K3 surfaces from \cite{Topiwala1987}.
In the K3 case, a Kähler metric is constructed explicitly, and it is shown that a nearby Calabi-Yau metric exists.
The question how far the explicit Kähler metric is from the Calabi-Yau metric has been studied extensively in \cite{Kobayashi1990,Donaldson2010,Jiang2025}.
Such estimates were then used in \cite{Lye2023} and \cite[Section 2.4]{Oliveira2023}.

In this article, we give a partial answer to the question how far away the torsion-free $G_2$-structure is from the $G_2$-structure with small torsion.
Because the construction from \cite{Joyce1996} is a generalisation of the original Kummer construction, our estimates for the seven-dimensional manifold can be adapted to imply estimates for the original Kummer construction.
We consider the case of $\Gamma=\Z^3$, where all orbifold singularities are resolved by gluing in Eguchi-Hanson spaces, and we denote its resolution by $N_t$.
Here, $0<t \ll 1$ is the gluing parameter that controls how big the glued in Eguchi-Hanson spaces are.
We prove an improved estimate for the difference between the torsion-free $G_2$-structure and the one with small torsion.
The main result is \cref{corollary:kummer-construction-simplified-torsion-free-estimate}:

\begin{theorem*}
 Choose $\alpha \in (0,1)$ and $\beta \in (-1,0)$ both close to $0$.
 Let $N_t$ be the resolution of $T^7/\Gamma$ from \cref{equation:resolution-of-t7-gamma} and $\varphi^t \in \Omega^3(N_t)$ the $G_2$-structure with small torsion from \cref{equation:g2-structure-on-kummer-construction}.
 There exists $c>0$ independent of $t$ such that the following is true:
 for $t$ small enough, there exists $\eta^t \in \Omega^2(N_t)$ such that $\tilde{\varphi}=\varphi^t+\d \eta^t$ is a torsion-free $G_2$-structure, and $\eta^t$ satisfies
 \[
  \|{
   \eta^t
  }_{C^{2,\alpha/2}_{\beta;t}}
  \leq
  ct^{7/2-\beta}.
 \]
 In particular,
 \[
 \|{\tilde{\varphi}-\varphi^t}_{L^\infty} \leq ct^{5/2} \text{ and }
 \|{\tilde{\varphi}-\varphi^t}_{C^{0,\alpha/2}} \leq ct^{5/2-\alpha/2} \text{ as well as }
 \|{\tilde{\varphi}-\varphi^t}_{C^{1,\alpha/2}} \leq ct^{3/2-\alpha/2}.
 \]
\end{theorem*}

Here, the norm 
$\|{
   \; \cdot \;
  }_{C^{2,\alpha/2}_{\beta;t}}$
is a weighted Hölder norm.
The norms in the last line of the theorem are ordinary, unweighted norms.
The group $\Gamma$ is a finite group acting through $G_2$-involutions on $T^7$.
In \cite{Joyce1996,Joyce2000} the estimate
 $\|{
   \tilde{\varphi}-\varphi
  }_{L^\infty}
  \leq
  ct^{1/2}$
 was shown.
 In this sense, the estimates from \cref{corollary:kummer-construction-simplified-torsion-free-estimate} are an improvement.
 The theorem hinges on an estimate for the inverse of the Laplacian acting on $2$-forms on the resolution of $T^7/\Gamma$.
 The crucial idea necessary for obtaining this estimate is to split $2$-forms into a part that is harmonic on the $4$-dimensional fibres orthogonal to the singular set of $T^7/\Gamma$, and a rest.
 The $4$-dimensional fibres are subsets of Eguchi-Hanson space $\XEH$, and the proof of \cref{corollary:kummer-construction-simplified-torsion-free-estimate} uses detailed knowledge of the harmonic forms on $\XEH$.
 The space $\XEH$ admits a harmonic $2$-form $\nu$ that can be written down explicitly and comes from rescaling the metric.
 In \cref{corollary:kernel-for-eguchi-laplacian}, we denote the Laplacian on $\XEH$ acting on $p$-forms by $\Delta_{p,g_{(1)}}$, and we prove that $\nu$ is essentially the only form with decay:
 
 \begin{theorem*}
  For $\lambda \in (-4,0)$, the $L^2_{2,\lambda}$-kernels of $\Delta_{p,g_{(1)}}$ acting on $p$-forms of different degrees are the same as the $L^2$-kernels, namely:
  \begin{align*}
   \Ker (\Delta_{g_{(1)}}:L^2_{2,\lambda}(\Lambda^2(\XEH))\rightarrow L^2_{0,\lambda-2}(\Lambda^2(\XEH)) )
   & = \< \nu \>,
   \\
   \Ker (\Delta_{g_{(1)}}:L^2_{2,\lambda}(\Lambda^p(\XEH))\rightarrow L^2_{0,\lambda-2}(\Lambda^p(\XEH)) )
   & = 0
   \text{ for }
   p \neq 2.
  \end{align*}
\end{theorem*}

Here $L^2_{2,\lambda}(\Lambda^p(\XEH))$ denote the usual weighted Sobolev spaces on asymptotically conical manifolds.
They consist of, roughly speaking, $L^2$-sections with $2$ weak derivatives that decay like $r^\lambda$ as $r \rightarrow \infty$, where $r$ is a radius function.

\textbf{Acknowledgments.}
This article was written during my PhD studies under the supervision of Jason Lotay and Simon Donaldson.
I am indebted to them for sharing their ideas and supporting me during my work on this thesis.
I thank the two anonymous reviewers for their very helpful comments.
Their comments included the proof to \cref{proposition:critical-dimension}.
This work was supported by [EP/L015234/1], the Engineering and Physical Sciences Research
Council, the EPSRC Centre for Doctoral Training in Geometry and Number Theory (the London School of Geometry and Number Theory), University College London. 
The author was also supported by Imperial College London.

\section{Background}

\subsection{Definition of the Eguchi-Hanson Space}

The singularities of the $G_2$-orbifolds we are interested in are locally modelled on $\R^3 \times \C^2/\{\pm 1\}$.
In order to resolve these singularities, we study the resolution of the point singularity of $\C^2/\{\pm 1\}$, called the Eguchi-Hanson space.
Some references for this space are \cite[Section 7.2]{Joyce2000} and \cite[Section 1]{Dancer1999}.
We begin by defining the Eguchi-Hanson space and the Eguchi-Hanson metrics, which are a $1$-dimensional family of Hyperkähler metrics, controlled by a parameter $k \in \R_{\geq 0}$.
For $k>0$ we get a metric on a smooth $4$-manifold (this is point one of the following proposition), and for $k=0$ we get the standard metric on $\mathbb{H}/\{ \pm 1\}$ or equivalently $\C^2/\{ \pm 1\}$ (this is point two of the following proposition).
The meaning of $k$ is the \emph{scale} of the Eguchi-Hanson space:
namely, for $k>0$ the space contains a minimal $2$-sphere whose diameter is proportional to $k^{1/4}$.
For $k=0$ one can think of the sphere having collapsed to size $0$.

\begin{proposition}
\label{lemma:the-eguchi-hanson-metric}
 Let $r$ be a coordinate on the $\R_{\geq 0}$-factor of $\R_{\geq 0} \times \SO(3)$.
 Let
 \begin{align*}
  \eta^1=
  2
  \begin{pmatrix}
   0&0&0\\
   0&0&1\\
   0&-1&0
  \end{pmatrix},
  \eta^2=
  2
  \begin{pmatrix}
   0&0&-1\\
   0&0&0\\
   1&0&0
  \end{pmatrix},
  \eta^3=
  2
  \begin{pmatrix}
   0&-1&0\\
   1&0&0\\
   0&0&0
  \end{pmatrix}
  \in \so(3)
 \end{align*}
 and denote the dual basis extended to left-invariant $1$-forms on $\SO(3)$ by the same symbols.
 For $k \geq 0$, let $f_k: \R_{>0} \times \SO(3) \rightarrow \R_{>0}$ be defined by $f_k(r)=(k+r^2)^{1/4}$ and set
 \begin{align*}
  \d t&=f_k^{-1}(r) \d r,
  &
  e^1(r)&=r f_k^{-1}(r) \eta^1,
  &
  e^2(r)&=f_k(r) \eta^2,
  &
  e^3(r)&=f_k(r) \eta^3.
 \end{align*}
 Define $\omega_1^{(k)}, \omega_2^{(k)}, \omega_3^{(k)} \in \Omega^2(\R_{>0} \times \SO(3))$ to be
 \begin{align}
  \label{equation:hyperkaehler-triple-on-EH}
  \omega_1^{(k)}&=\d t \wedge e^1+e^2 \wedge e^3, &
  \omega_2^{(k)}&=\d t \wedge e^2+e^3 \wedge e^1, &
  \omega_3^{(k)}&=\d t \wedge e^3+e^1 \wedge e^2,
 \end{align}
 and denote by $g_{(k)}$ the metric on $\R_{>0} \times \SO(3)$ that makes $(\d t, e^1, e^2, e^3)$ an orthonormal basis.
 \begin{enumerate}
 \item
 If $k>0$, consider the copy of $\SO(2)$ in $\SO(3)$ defined by $\{ \exp(s \cdot \eta^1): s \in \R \}$, defining a right action of $\SO(2)$ on $\SO(3)$.
 Denote by $V \simeq \R^2$ the standard representation of $\SO(2)$.
 Define $\Psi: \SO(3) \times \R_{>0} \rightarrow \SO(3) \times V$ as $\Psi(g,r)=(g,(r,0))$.
 Denote
 \begin{align*}
  \XEH
  =
  \SO(3)
  \times_{\SO(2)}
  V.
 \end{align*}
 Then $\Psi$ induces a smooth injective map $\hat{\Psi}:\SO(3) \times \R_{>0} \rightarrow \XEH$ that is a diffeomorphism onto its image, and the forms $\hat{\Psi}_*(\omega_i^{(k)})$ can be extended to smooth 2-forms on all of $\XEH$.
 Furthermore, $\hat{\Psi}_*(g_{(k)})$ can also be extended to a metric on all of $\XEH$, and $(\XEH, \hat{\Psi}_*(g_{(k)}))$ is a Hyperkähler manifold.
 
 \item
 If $k=0$:
 parametrise the quaternions as $x_0+x_1i+x_2j+x_3k$ with $x_0,x_1,x_2,x_3 \in \R$, embed $S^3 \subset \mathbb{H}$ as the unit sphere, and fix the identification $\phi:S^3/\{\pm 1\} \rightarrow \SO(3)$ that maps $x$ onto the map $y \mapsto x \cdot y \cdot x^{-1}$, where we use $S^3 /\{\pm 1\} \subset \mathbb{H}/\{\pm 1\}$ and $\cdot$ denotes quaternionic multiplication, for $x \in S^3/\{\pm 1\} \subset \mathbb{H}/\{\pm 1\}$.
 Denote 
 \begin{align*}
  \Phi:
  \SO(3) \times \R_{>0} &\rightarrow \mathbb{H}/\{\pm 1\}
  \\
  (x,r) & \mapsto \sqrt{r} \cdot \phi^{-1}(x).
 \end{align*}
 Then $\Phi^* {\omega}_i = \omega_i^{(0)}$ for $i \in \{1,2,3\}$ and $\Phi^* g = g_{(0)}$, where $g, \omega_1,\omega_2,\omega_3 \in \Omega^2(\mathbb{H})$ are the metric and standard Hyperkähler triple on $\mathbb{H}$.
 \end{enumerate}
\end{proposition}

The space $\XEH$ is the total space of a vector bundle over $S^2=\SO(3)/\SO(2)$.
The image of $\hat{\Psi}$ is the complement of the zero section of this vector bundle.
By slight abuse of notation, we will denote the extensions of $\omega_i^{(k)}$ for $i \in \{1,2,3\}$ and $g_{(k)}$ to $\XEH$ in the case $k>0$ by the same symbol, suppressing the pushforward under $\hat{\Psi}$.

\begin{proof}
 For $k>0$:
 the fact that $\omega_1^{(k)},\omega_2^{(k)},\omega_3^{(k)},g_{(k)}$ can be extended to all of $\XEH$ was proven, for example, in \cite[Section 2.4]{Lotay2017}.
 One checks using a direct computation that $\omega_i^{(k)}$ for $i \in \{1,2,3\}$ is closed and \cite[Lemma 6.8]{Hitchin1987} implies that $\omega_i^{(k)}$ is also parallel for $i \in \{1,2,3\}$.
 Both the symplectic forms and the metric are defined using the same orthonormal basis, which proves that they are compatible.
 The case $k=0$ is a direct calculation.
\end{proof}

\begin{remark}
A possible point of confusion is that the function $r:\XEH \rightarrow \R$ is (for large $r$) approximately the squared distance to the bolt $\SO(3) \times_{SO(2)} \{0\}$ of $\XEH$, so it is not a radius function.
\end{remark}

The Hyperkähler structure on $\XEH$ also has the important property that it approximates the flat Hyperkähler structure on $\mathbb{H}$ for large values of $r$.
The following definition makes this notion precise, and \cref{lemma:eguchi-hanson-ALE-condition} states that the Hyperkähler structure on $\XEH$ does indeed have this property.

\begin{definition}[Definition 7.2.1 in \cite{Joyce2000}]
\label{definition:ale}
 Let $G$ be a finite subgroup of $\Sp(1)$, and let 
 $(\hat{\omega}_1,\hat{\omega}_2,\hat{\omega}_3,\hat{g})$
 be the Euclidean Hyperkähler structure on $\mathbb{H}$, and 
 $\sigma: \mathbb{H}/G \rightarrow [0,\infty)$ 
 the radius function on $\mathbb{H}/G$.
 We say that a Hyperkähler $4$-manifold $(X,\omega_1,\omega_2,\omega_3,g)$ is \emph{asymptotically locally Euclidean (ALE) asymptotic to $\mathbb{H}/G$}, if there exists a compact subset $S \subset X$ and a map $\pi:X \setminus S \rightarrow \mathbb{H}/G$ that is a diffeomorphism between $X \setminus S$ and $\{x \in \mathbb{H}/G : \sigma(x)>R \}$ for some $R>0$, such that
 \begin{align}
  \hat{\nabla}^k(\pi_*(g)-\hat{g})=\mathcal{O}(\sigma^{-4-k})
  \text{ and }
  \hat{\nabla}^k(\pi_*(\omega_i)-\hat{\omega}_i)=\mathcal{O}(\sigma^{-4-k})
 \end{align}
 as $\sigma \rightarrow \infty$, for $i \in \{1,2,3\}$ and $k \geq 0$, where $\hat{\nabla}$ is the Levi-Civita connection of $\hat{g}$.
\end{definition}

\begin{proposition}[Example 7.2.2 in \cite{Joyce2000}]
\label{lemma:eguchi-hanson-ALE-condition}
  The $1$-form $\tau_1^{(k)} \in \Omega^1(\XEH \setminus \SO(3) \times_{\SO(2)} \{0 \})$ given by $\tau_1^{(k)}=(f_k^2-f_0^2) \eta^1$ satisfies $\omega_1^{(k)}-\omega_1^{(0)} = \d \tau_1^{(k)}$ and for any $l \in \Z$
  \begin{align}
  \label{equation:ALE-hyperkaehler-decay}
   \left|
   \nabla^l
   \tau_1^{(1)}
   \right|_{g_{(0)}}
   &=
   \mathcal{O}(r^{-3-l}),
  \end{align}
  where $\nabla$ denotes the Levi-Civita connection of $g_{(0)}$.
  Furthermore, $\omega_2^{(k)}-\omega_2^{(0)}=0$, and $\omega_3^{(k)}-\omega_3^{(0)}=0$.
  In particular, $(\XEH, \omega_1^{(k)}, \omega_2^{(k)}, \omega_3^{(k)}, g_{(k)})$ is ALE asymptotic to $\mathbb{H}/\{\pm 1\}$.
\end{proposition}

\begin{remark}
\label{lemma:eguchi-hanson-blowup-map}
\label{equation:blowup-map-definition}
By definition, $\XEH$ is an associated bundle over $\SO(3)/\SO(2)=S^2$.
In fact, $\XEH$ is diffeomorphic to the total space of $T^* S^2$, which itself is diffeomorphic to $T^* \CP^1$.
It is a folklore result that $(\XEH, J^{(k)}_1)$ is biholomorphic to $T^* \CP^1$ for all $k>0$, which in turn is the blowup of $\C^2/\{\pm 1\}$ in the origin, see e.g. \cite[p. 17]{Dancer1999} for the statement.
Explicitly, the blowup map $\rho: \XEH \rightarrow \mathbb{H}/\{\pm 1\}$ is given by setting $\rho=\Phi \circ \hat{\Psi}$ on the complement of the zero section in $\XEH$ and mapping the zero section to $[0] \in \mathbb{H}/\{ \pm 1\}$.
This map satisfies $r=|\rho|^2$.
\end{remark}

\subsection{$G_2$-structures}
\label{section:g2-structures}

In this section we collect standard facts about $G_2$-geometry needed later.

\begin{definition}[Definition 10.1.1 in \cite{Joyce2000}]
\label{definition:g2}
 Let $(x_1,\dots,x_7)$ be coordinates on $\R^7$.
 Write $\d x_{ij \dots l}$ for the exterior form $\d x_i \wedge \d x_j \wedge \dots \wedge \d x_l$.
 Define $\varphi_0 \in \Omega^3(\R^7)$ by
 \begin{align}
 \label{equation:standard-g2-form}
  \varphi_0
  =
  \d x_{123}+\d x_{145}+\d x_{167}+\d x_{246}-\d x_{257}-\d x_{347}-\d x_{356}.
 \end{align}
 The subgroup of $\GL(7,\R)$ preserving $\varphi_0$ is the exceptional Lie group $G_2$.
 It also fixes the Euclidean metric $g_0=\d x_1^2+\dots + \d x_7^2$, the orientation on $\R^7$, and $* \varphi_0 \in \Omega^4(\R^7)$.
\end{definition}

On $\mathbb{H}$ with coordinates $(y_0,y_1,y_2,y_3)$ we have the three symplectic forms $\omega_1,\omega_2,\omega_3$ given as
\begin{align*}
 \omega_1&=
 \d y_0 \wedge \d y_1 + \d y_2 \wedge \d y_3,
 &
 \omega_2&=
 \d y_0 \wedge \d y_2 - \d y_1 \wedge \d y_3,
 &
 \omega_3&=
 \d y_0 \wedge \d y_3 + \d y_1 \wedge \d y_2.
\end{align*}
Identify $\R^7$ with coordinates $(x_1,\dots,x_7)$ with $\R^3 \oplus \mathbb{H}$ with coordinates $((x_1,x_2,x_3),(y_1,y_2,y_3,y_4))$.
Then we have for $\varphi_0, * \varphi_0$ from \cref{definition:g2}:
\begin{align}
\label{equation:product-g2-structure}
 \varphi_0
 &=
 \d x_{123} - 
 \sum_{i=1}^3
 \d x_i \wedge \omega_i
 ,&
 *\varphi_0
 &=
 \vol_{\mathbb{H}}-
 \sum_
 {\substack{(i,j,k) = (1,2,3)\\
 \text{and cyclic permutation}}}
 \omega_i \wedge \d x_{jk}.
\end{align}

This linear algebra statement easily extends to product manifolds in the following sense:
if $X$ is a Hyperkähler $4$-manifold, and $\R^3$ is endowed with the Euclidean metric, then $\R^3 \times X$ has a $G_2$-structure.
The $G_2$-structure is given by the same formula as in the flat case, namely \cref{equation:product-g2-structure}, after replacing $(\omega_1,\omega_2,\omega_3)$ with the triple of parallel symplectic forms defining the Hyperkähler structure on $X$.

\begin{definition}
 Let $M$ be an oriented $7$-manifold.
 A principal subbundle $Q$ of the bundle of oriented frames with structure group $G_2$ is called a \emph{$G_2$-structure}.
 Viewing $Q$ as a set of linear maps from tangent spaces of $M$ to $\R^7$, there exists a unique $\varphi \in \Omega^3(M)$ such that $Q$ identifies $\varphi$ with $\varphi_0 \in \Omega^3(\R^7)$ at every point.
 
 Such $G_2$-structures are in $1$-$1$ correspondence with $3$-forms on $M$ for which there exists an oriented isomorphism mapping them to $\varphi_0$ at every point.
 We will therefore also refer to such $3$-forms as $G_2$-structures.
\end{definition}

Let $M$ be a manifold with $G_2$-structure $\varphi$.
We call $\nabla \varphi$ the \emph{torsion} of a $G_2$-structure $\varphi \in \Omega^3(M)$.
Here, $\nabla$ denotes the Levi-Civita induced by $\varphi$ in the following sense:
we have $G_2 \subset \SO(7)$, so $\varphi$ defines a Riemannian metric $g$ on $M$, which in turn defines a Levi-Civita connection.
To emphasise non-linearity, we also use the following notation:
write $\Theta(\varphi)=*\varphi$, where ``$*$'' denotes the Hodge star defined by $g$.
Using this, the following theorem gives a characterisation of torsion-free $G_2$-manifolds:

\begin{theorem}[Propositions 10.1.3 and 10.1.5 in \cite{Joyce2000}]
\label{theorem:torsion-free-g2-structures-characterisation}
 Let $M$ be an oriented $7$-manifold with $G_2$-structure $\varphi$ with induced metric $g$.
 The following are equivalent:
 \begin{enumerate}[label=(\roman*)]
  \item 
  $\Hol(g) \subseteq G_2$,
  
  \item
  $\nabla \varphi=0$ on $M$, where $\nabla$ is the Levi-Civita connection of $g$, and
  
  \item
  $\d \varphi =0$ and $\d \Theta(\varphi)=0$ on $M$.
 \end{enumerate}
 If these hold then $g$ is Ricci-flat.
\end{theorem}

Later on, we will investigate perturbations of $G_2$-structures and analyse how they affect the torsion.
To this end, we will use the following estimates for the map $\Theta$ defined before:

\begin{proposition}[Proposition 10.3.5 in \cite{Joyce2000} and eqn. (21) of part I in \cite{Joyce1996}]
\label{proposition:Theta-estimates}
 There exists $\epsilon > 0$ and $c>0$ such that whenever $M$ is a $7$-manifold with $G_2$-structure $\varphi$ satisfying $\d \varphi=0$, then the following is true.
 Suppose $\chi \in C^\infty(\Lambda^3 T^*M)$ and $| \chi | \leq \epsilon$.
 Then $\varphi+\chi$ is a $G_2$-structure, and
 \begin{align}
  \label{equation:theta-expansion}
  \Theta(\varphi+\chi)
  =*\varphi-T(\chi)-F(\chi),
 \end{align}
 where ``$*$'' denotes the Hodge star with respect to the metric induced by $\varphi$, $T: \Omega^3(M) \rightarrow \Omega^4(M)$ is a linear map (depending on $\varphi$), and $F$ is a smooth function from the closed ball of radius $\epsilon$ in $\Lambda^3T^*M$ to $\Lambda^4 T^*M$ with $F(0)=0$.
 Furthermore,
  \begin{align*}
   \left|
    F(\chi)
   \right|
   &\leq
   c
   \left| \chi \right|^2,
   \\
   \left|
    \d \,(F(\chi))
   \right|
   &\leq
   c
   \left\{
    \left| \chi \right|^2 \left| \d^* \varphi \right|
    +
    \left| \nabla \chi \right| \left| \chi \right|
   \right\},
   \\
   [\d \, (F(\chi))]_\alpha
   &\leq
   c
   \left\{
    [\chi]_\alpha \|{ \chi }_{L^\infty} \|{ \d^* \varphi }_{L^\infty}
    +
    \|{ \chi }_{L^\infty}^2 [ \d ^* \varphi ]_\alpha
    +
    [\nabla \chi ]_\alpha \|{ \chi }_{L^\infty}
    +
    \|{ \nabla \chi }_{L^\infty} [\chi]_\alpha
   \right\},
  \end{align*}
  as well as
  \begin{align*}
   \left|
    \nabla (F(\chi))
   \right|
   &\leq
   c
   \left\{
    \left| \chi \right|^2 \left| \nabla \varphi \right|
    +
    \left| \nabla \chi \right| \left| \chi \right|
   \right\},
   \\
   [\nabla (F(\chi))]_{\alpha}
   &\leq
   c
   \left\{
    [\chi]_\alpha \|{ \chi }_{L^\infty} \|{ \nabla \varphi }_{L^\infty}
    +
    \|{ \chi }_{L^\infty}^2 [ \nabla \varphi ]_\alpha
    +
    [\nabla \chi ]_\alpha \|{ \chi }_{L^\infty}
    +
    \|{ \nabla \chi }_{L^\infty} [\chi]_\alpha
   \right\}.
  \end{align*}
  Here, $\left| \cdot \right|$ denotes the norm induced by $\varphi$, $\nabla$ denotes the Levi-Civita connection of this metric, and $[ \cdot ]_{\alpha}$ denotes the unweighted Hölder semi-norm induced by this metric.
\end{proposition}

Finally, the landmark result on the existence of torsion-free $G_2$-structures is the following theorem.
It first appeared in \cite[part I, Theorem A]{Joyce1996}, and we present a rewritten version in analogy with \cite[Theorem 2.7]{Joyce2017}:

\begin{theorem}
\label{theorem:original-torsion-free-existence-theorem}
 Let $\alpha, K_1, K_2, K_3$ be any positive constants.
 Then there exist $\epsilon \in (0,1]$ and $K_4 >0$, such that whenever $0<t \leq \epsilon$, the following holds.
 
 Let $M$ be a compact oriented $7$-manifold, with $G_2$-structure $\varphi$ with induced metric $g$ satisfying $\d \varphi=0$.
 Suppose there is a closed $3$-form $\psi$ on $M$ such that $\d^* \varphi=\d^* \psi$ and
 \begin{enumerate}[label=(\roman*)]
  \item 
  $\|{ \psi }_{C^0} \leq K_1 t^\alpha$,
  $\|{ \psi }_{L^2} \leq K_1 t^{7/2+\alpha}$,
  and
  $\|{ \psi }_{L^{14}} \leq K_1 t^{-1/2+\alpha}$.
  
  \item
  The injectivity radius $\inj$ of $g$ satisfies $\inj \geq K_2 t$.
  
  \item
  The Riemann curvature tensor $\Rm$ of $g$ satisfies $\|{ \Rm }_{C^0} \leq K_3 t^{-2}$.
 \end{enumerate}
 Then there exists a smooth, torsion-free $G_2$-structure $\tilde{\varphi}$ on $M$ such that $\|{ \tilde{\varphi}-\varphi}_{C^0} \leq K_4 t^\alpha$ and $[\tilde{\varphi}]=[\varphi]$ in $H^3(M,\R)$.
 Here all norms are computed using the original metric $g$.
\end{theorem}

The main purpose of \cref{section:torsion-free-structures-on-the-generalised-kummer-construction} will be to prove an improved existence theorem, specialised to the resolution of $T^7/\Gamma$.
This will be achieved in \cref{theorem:torsion-free-final-existence-theorem}.

\section{Harmonic Forms with Decay on the Eguchi-Hanson Space}
\label{section:analysis-on-eguchi-hanson}

The aim of this section is to prove \cref{corollary:kernel-for-eguchi-laplacian}.
That is, to prove that there is only one harmonic form on Eguchi-Hanson space that decays at infinity, up to scaling.
We will achieve this using the techniques of Lockhart and McOwen (cf. \cite{Lockhart1985,Lockhart1987}), which give a description of the harmonic forms on asymptotically conical manifolds, depending on information about harmonic forms on the asymptotic cone.
To this end, we begin by studying the asymptotic cone of Eguchi-Hanson space $\XEH$, namely the cone over $\SO(3)$. 

\subsection{Harmonic Forms on $(\mathbb{C}^2\setminus \{0\})/\{ \pm 1\}$}

In this section, we will list homogeneous harmonic forms on $(\mathbb{C}^2\setminus \{0\})/\{ \pm 1\}$ with decay.
Because $(\mathbb{C}^2\setminus \{0\})/\{ \pm 1\}$ is the cone over $\SO(3)$, we will see that such forms correspond to eigenforms on $\SO(3)$, and we will therefore review the spectral decomposition of the Laplacian on $S^3$ and $\SO(3)$.

We begin by defining cones and homogeneous forms on them.

\begin{definition}
 For a Riemannian manifold $(\Sigma, g_\Sigma)$, the Riemannian manifold $C(\Sigma)=\Sigma \times \R_{> 0}$ endowed with the metric
 $
  g_C=
  \d r^2+r^2 g_\Sigma
 $
 is called the \emph{cone over $\Sigma$}.
\end{definition}

\begin{definition}
 Let $\lambda \in \R$.
 Then $\gamma \in \Omega^k(C(\Sigma))$ is called \emph{homogeneous of order $\lambda$} if there exist $\alpha \in \Omega^{k-1}(\Sigma), \beta \in \Omega^k(\Sigma)$ such that
 \[
  \gamma = r^{\lambda+k}\left( \frac{dr}{r}\wedge\alpha + \beta\right).
 \]
\end{definition}

\begin{remark}
 For $t \in \R_{>0}$ denote by $(\cdot t): C(\Sigma) \rightarrow C(\Sigma)$ the dilation map given by $(\cdot t)(r,\sigma)=(tr,\sigma)$ for $(r,\sigma) \in C(\Sigma)$.
 Then, if $\gamma \in \Omega^k(C(\Sigma))$ is homogeneous of order $\lambda$, we have $(\cdot t)^* |\gamma |_{g_C}=t^\lambda | \gamma |_{g_C}$.
\end{remark}

Homogeneous harmonic forms do not exist for all orders and we make the following definition:

\begin{definition}
 For a cone $C=C(\Sigma)$, denote by $\Delta_{k,\Sigma}$ and $\Delta_{k,C}$ the Laplacian acting on $k$-forms on $\Sigma$ and $C$ respectively.
 The set
 \[
  \mathcal{D}_{\Delta_{k,C}}
  =
  \{
   \lambda \in \R:
   \exists \gamma \in \Omega^k(C), \gamma \neq 0, \text{ homogeneous of order } \lambda \text{ with } \Delta_{k,C} \gamma = 0
  \}
  \]
  is called the set of \emph{critical rates of $\Delta_{k,C}$}.
\end{definition}

It will turn out that critical rates are intimately related to harmonic forms on Eguchi-Hanson space.
This is the content of the next subsection and we will see the set $\mathcal{D}_{\Delta_{k,C}}$ appear again there.
The purpose of the current subsection is to describe $\mathcal{D}_{\Delta_{1,C(\SO(3))}}$ and $\mathcal{D}_{\Delta_{2,C(\SO(3))}}$, which is achieved in \cref{proposition:harmonic-on-eguchi-cone}.
We prepare the proposition by putting some results for harmonic forms on Riemannian cones in place:

\begin{lemma}[Lemma A.1 in \cite{Foscolo2020}]
\label{lemma:laplacian-on-cone-explicit}
Let $\gamma = r^{\lambda+k}\left( \frac{dr}{r}\wedge\alpha + \beta\right)$ be a $k$-form on $C(\Sigma)$ homogeneous of order $\lambda$.
For every function $u=u(r)$ we have $\Delta_{k,C} (u\gamma) = r^{\lambda+k-2}\left( \tfrac{dr}{r}\wedge A + B\right)$, where
\begin{align*}
A &= u\Big( \Delta_{k-1,\Sigma} \alpha - (\lambda+k-2)(\lambda+n-k)\alpha -2d^\ast\beta \Big) -r\dot{u}\left( 2\lambda +n-1\right) \alpha - r^2 \ddot{u}\,\alpha,
\\
B &= u\Big( \Delta_{k,\Sigma} \beta - (\lambda+n-k-2)(\lambda+k)\beta -2d\alpha \Big) -r\dot{u}\left( 2\lambda +n-1\right) \beta - r^2 \ddot{u}\,\beta.
\end{align*}
We also used the shorthand notation $\dot{u}=\frac{\d}{\d r}u$ and $\ddot{u}=\frac{\d^2}{\d r^2}u$.
\end{lemma}

\begin{theorem}[Theorem A.2 in \cite{Foscolo2020}]
\label{theorem:foscolo-harmonic-forms-on-cone}
Let $\gamma = r^{\lambda+k}\left( \frac{dr}{r}\wedge\alpha + \beta\right)$ be a harmonic $k$-form on $C(\Sigma)$ homogeneous of order $\lambda$.
Then $\gamma$ decomposes into the sum of homogeneous harmonic forms $\gamma = \gamma_1 + \gamma_2 + \gamma_3 + \gamma_4$ where $\gamma_i = r^{\lambda+k}\left( \frac{dr}{r}\wedge\alpha_i + \beta_i\right)$ satisfies the following conditions.
\begin{enumerate}[label=(\roman*)]
\item 
$\beta_1=0$ and $\alpha_1$ satisfies $d\alpha_1=0$ and $\Delta_{k-1,\Sigma}
\alpha_1 = (\lambda+k-2)(\lambda+n-k)\alpha_1$.

\item 
$(\alpha_2,\beta_2) \in\Omega^{k-1}_{coexact} (\Sigma) \times\Omega^k_{exact}(\Sigma)$ satisfies the first-order system
\[
d\alpha_2 = (\lambda+k)\beta_2, \qquad d^*\beta_2 = (\lambda+n-k)\alpha_2.
\]
In particular, if $(\alpha_2,\beta_2)\neq 0$ then $\lambda+k\neq 0\neq \lambda+n-k$ and the pair $(\alpha_2,\beta_2)$ is uniquely determined by either of the two factors, which is a coexact/exact eigenform of the Laplacian with eigenvalue $(\lambda+k)(\lambda+n-k)$.

\item 
$(\alpha_3,\beta_3) \in\Omega^{k-1}_{coexact}(\Sigma) \times \Omega^k_{exact}(\Sigma)$ satisfies the first-order system
\[
d\alpha_3 + (\lambda+n-k-2)\beta_3=0=d^*\beta_3 + (\lambda+k-2)\alpha_3.
\]
In particular, if $(\alpha_3,\beta_3)\neq 0$ then $\lambda+k-2\neq 0\neq \lambda+n-k-2$ and the pair $(\alpha_3,\beta_3)$ is uniquely determined by either of the two factors, which is a coexact/exact eigenform of the Laplacian with eigenvalue $(\lambda+k-2)(\lambda+n-k-2)$.

\item 
$\alpha_4=0$ and $\beta_4$ satisfies $d^* \beta_4=0$ and $\Delta_{k,\Sigma} \beta_4 = (\lambda+n-k-2)(\lambda+k)\beta_4$.
\end{enumerate}
The decomposition $\gamma = \gamma_1 + \gamma_2 + \gamma_3 + \gamma_4$ is unique, except when $\lambda = -\frac{n-2}{2}$; in that case forms of type (ii) and (iii) coincide, and there is a unique decomposition $\gamma = \gamma_1 + \gamma_2 + \gamma_4$.
\end{theorem}

The previous proposition relates harmonic forms on the cone $C(\SO(3))$ to eigenforms of the Laplacian on $\SO(3)$.
The group $\SO(4)$ acts via pullback on complex-valued differential forms on $S^3$, and it turns out that the decomposition of this action into irreducible components gives the spectral decomposition for the Laplacian on $S^3$.
This is made precise in the following two theorems, and as $S^3$ is a double cover of $\SO(3)$, we will get the spectral decomposition of the Laplacian on $\SO(3)$ from them.

\begin{theorem}[Theorem B in \cite{Folland1989}]
\label{theorem:folland-theorem-B}
 The complex-valued $L^2$-functions and $1$-forms on $S^3$ decompose into the following irreducible $\SO(4)$-invariant subspaces:
 \begin{align*}
  \Omega^0(S^3, \C)
  &=
  \bigoplus _{m=1}^\infty
  \Phi_{0,m},
  \\
  \Omega^1(S^3, \C)
  &=
  \bigoplus _{m=1}^\infty
  \left(
   \Phi_{1,m} \oplus
   \Phi_{1,m}^- \oplus
   \Psi_{1,m}
  \right).
 \end{align*}
 Here, $\Phi_{0,m}$, $\Phi_{1,m}, \Phi_{1,m}^-, \Psi_{1,m}$ are defined as follows:
 denote by $j: S^3 \rightarrow \R^4$ the inclusion map and define $z_1=x_1+i x_2$, $z_2=x_3+ix_4$, and $\partial r = \sum_{j=1}^4 x_j \partial x_j$.
 Then let
 \begin{align*}
  \Phi_{0,m}
  &=
  j^* \mathscr{G}_{0,m+1},
  \text{ where $\mathscr{G}_{0,m}$ is the smallest $\SO(4)$-inv. space containing $z_1^{m-1}$},
  \\
  \Phi_{1,m}
  &=
  j^* \mathscr{F}_{1,m},
  \text{ where $\mathscr{F}_{1,m}$ is the smallest $\SO(4)$-inv. space containing $z_1^{m-1} \partial r \lrcorner (\d z_1 \wedge \d z_2)$}.
  \\
  \Phi_{1,m}^-
  &=
  j^* \mathscr{F}^-_{1,m},
  \text{ where $\mathscr{F}^-_{1,m}$ is the smallest $\SO(4)$-inv. space containing $z_1^{m-1} \partial r \lrcorner (\d z_1 \wedge \d \overline{z_2})$}.
  \\
  \Psi_{1,m}
  &=
  j^* \mathscr{G}_{1,m},
  \text{ where $\mathscr{G}_{1,m}$ is the smallest $\SO(4)$-inv. space containing $z_1^{m-1} \d z_1$}.
 \end{align*}
\end{theorem}

\begin{theorem}[Theorem C in \cite{Folland1989}]
\label{theorem:folland-theorem-C}
 Let $\Phi_{0,m}, \Phi_{1,m}, \Phi_{1,m}^-, \Psi_{1,m}$ as in \cref{theorem:folland-theorem-B}.
 Then 
 \begin{itemize}
     \item 
     $\Phi_{0,m}$ is an eigenspace for the Laplacian with eigenvalue $m(m+2)$,

     \item 
     $\Phi_{1,m} \oplus \Phi_{1,m}^-$ is an eigenspace for the Laplacian with eigenvalue $(m+1)^2$,

     \item 
     $\Psi_{1,m}$ is an eigenspace for the Laplacian with eigenvalue $m(m+2)$.
 \end{itemize}
\end{theorem}

\begin{corollary}
\label{theorem:spectrum-of-laplacian-on-so3}
\label{theorem:spectrum-of-laplacian-on-1-forms}
 Let $S^3$ be endowed with the round metric and $\SO(3)=S^3/\{\pm 1 \}$ be endowed with the quotient metric.
 Then:
 \begin{enumerate}
 \item
 The spectrum of the Laplacian $\Delta_{0,\SO(3)}$ acting on real-valued $L^2$-functions on $\SO(3)$ is:
 \begin{align*}
  \Spec (\Delta_{0,\SO(3)})
  &=
  \{
   k(k+2):
   k \in \mathbb{Z}_{\geq 0},
   k \text{ even}
  \}
  =
  \{0,8,24,\dots\}.
 \end{align*}
 
 \item
 The smallest eigenvalue of the Laplacian $\Delta_{1,\SO(3)}$ acting on real-valued $1$-forms with coefficients in $L^2$ on $\SO(3)$ is $4$ and has multiplicity $6$.
 \end{enumerate}
\end{corollary}

\begin{proof}[Proof of \cref{theorem:spectrum-of-laplacian-on-so3}]
 \leavevmode
 \begin{enumerate}
 \item
 This follows from \cref{theorem:folland-theorem-B,theorem:folland-theorem-C} and the fact that functions in the space $\Phi_{0,m}$ defined in \cref{theorem:folland-theorem-B} are invariant under the antipodal map $(-1):S^3 \rightarrow S^3$ if and only if $m$ is even.
 
 \item
 By \cref{theorem:folland-theorem-C}, the smallest eigenvalue of the Laplacian acting on complex-valued $1$-forms on $S^3$ is $3$.
 The eigenforms in $\Psi_{1,1}$ are the differential of the linear functions in $\Phi_{0,1}$ and therefore not invariant under the antipodal map.
 Thus, the eigenvalue $3$ does not occur on $\SO(3)$.
 
 The next smallest eigenvalue is $4$.
 It is realised, and it remains to check the dimension of its eigenspace:
 for the complex vector spaces defined in \cref{theorem:folland-theorem-B} we have 
 $\Phi_{1,1} \simeq \left( \Lambda^2_+ \right)^{\mathbb{C}}$ 
 and 
 $\Phi_{1,1}^- \simeq \left( \Lambda^2_- \right)^{\mathbb{C}}$, 
 the complexification of (anti-)self-dual constant forms on $\R^4$.
 Here is how to see that $\Phi_{1,1} \simeq \left( \Lambda^2_+ \right)^{\mathbb{C}}$, the other isomorphism is analogous.
 We have 
 \[\d z_1 \wedge \d z_2 = \d x_{13}-\d x_{24}+i \d x_{23}+i \d x_{14} =: \omega.\]
 The element
 $g=\begin{pmatrix}
 0&1&0&0\\
 1&0&0&0\\
 0&0&0&1\\
 0&0&1&0
 \end{pmatrix} \in \SO(4)$
 sends this to $-\d x_{13}+\d x_{24}+i \d x_{23}+i \d x_{14}$, so the smallest $\SO(4)$-invariant space containing $\omega$ must also contain the self-dual form $\d x_{13}-\d x_{24}=\frac{1}{2}(\omega-g \omega)$.
 Because $\Lambda^2_+$ is irreducible, this $\SO(4)$-invariant space must contain all of $(\Lambda^2_+)^{\C}$.
 Contracting with the radial vector field $\partial r$ and restricting to $S^3$ are $\SO(4)$-equivariant operations, one checks that the result is non-zero, and therefore $\Phi_{1,1} \simeq \left( \Lambda^2_+ \right)^{\mathbb{C}}$.
 
 Altogether, $\Phi_{1,1}$ and $\Phi^-_{1,1}$ are representations of $\SO(4)$ of complex dimension $3$.
 They consist of $1$-forms on $S^3$ that are invariant under the antipodal map, which proves the claim. \qedhere
 \end{enumerate}
\end{proof}

It also follows from \cite[Theorem 7.6]{Cheeger1994} together with the Hodge decomposition and the first part of \cref{theorem:spectrum-of-laplacian-on-so3} that the smallest eigenvalue of the Laplacian $\Delta_{1,\SO(3)}$ is $3$.

We can now combine the results about harmonic forms on $C(\SO(3))$ with the spectral decomposition of the Laplacian on $\SO(3)$ to find the critical rates for the Laplacian on $C(\SO(3))$.
The space of covariant constant $2$-forms on $\C^2$ is six-dimensional and one may multiply each such form with the fundamental solution of the Laplace equation $r^{-2}$ to obtain a six-dimensional space of harmonic $2$-forms with rate $-2$.
The following proposition states that there are no other harmonic $2$-forms or $1$-forms up to rate $0$:

\begin{proposition}
\label{proposition:harmonic-on-eguchi-cone}
 \leavevmode
 \begin{enumerate}
  \item
  There are no harmonic $1$-forms on $(\mathbb{C}^2\setminus \{0\})/\{ \pm 1\}$ which are homogeneous of order $\lambda$ for $\lambda \in [-2,0)$.
  In other words $\mathcal{D}_{\Delta_{1,(\mathbb{C}^2\setminus \{0\})/\{ \pm 1\}}} \cap [-2,0) = \emptyset$.
  
  \item
  There is a six-dimensional space of harmonic $2$-forms on $(\mathbb{C}^2\setminus \{0\})/\{ \pm 1\}$ which are homogeneous of order $-2$.
  
  There are no harmonic $2$-forms on $(\mathbb{C}^2\setminus \{0\})/\{ \pm 1\}$ which are homogeneous of order $\lambda$ for $\lambda \in (-2,0)$.
 \end{enumerate}
\end{proposition}

\begin{proof}
 It follows from point two in \cref{lemma:the-eguchi-hanson-metric} that $C(\SO(3))$ and $(\mathbb{C}^2\setminus \{0\})/\{ \pm 1\}$ are isometric as Riemannian manifolds and we prove the statements on $C(\SO(3))$.
 \begin{enumerate}
  \item
  Let $\lambda \in [-2,0)$ and assume there exists a harmonic homogeneous $1$-form of order $\lambda$ on $C(\SO(3))$.
  We show that the $1$-form must vanish by showing that forms satisfying any of the cases (i), (ii), (iii), and (iv) from \cref{theorem:foscolo-harmonic-forms-on-cone} are zero.
  Using the notation from the theorem, we get the following:
  
  \begin{enumerate}[label=(\roman*)]
   \item
   In this case, $\Delta \alpha_1=(\lambda-1)(\lambda+3)\alpha_1$.
  For $\lambda \in [-2,0)$, the factor $(\lambda-1)(\lambda+3)$ is negative, so our assumption implies that $\alpha_1$ is a closed $0$-form that is an eigenform of $\Delta_{\SO(3)}$ for a negative eigenvalue, which implies $\alpha_1=0$ by \cref{theorem:spectrum-of-laplacian-on-so3}.
  
  \item
  In this case, $\beta_2$ is an exact $1$-form with $\Delta_{\SO(3)} \beta_2=(\lambda+1)(\lambda+3)\beta_2$.
  We have $(\lambda+1)(\lambda+3)<8$ for $\lambda \in [-2,0)$, and therefore $\beta_2=0$ as in case (i).
  
  \item
  In this case, $\beta_3$ is an exact $1$-form with $\Delta_{\SO(3)} \beta_3=(\lambda+1)(\lambda-3)\beta_3$, and $\beta_3=0$ follows as before.
  
  \item
  In this case, $\beta_4$ is a co-closed $1$-form with $\Delta_{\SO(3)} \beta_3=(\lambda+1)^2\beta_3$.
  For $\lambda \in [-2,0)$, we have $(\lambda+1)^2 < 3$, and because of \cref{theorem:spectrum-of-laplacian-on-1-forms} this implies $\beta_4=0$.  
  \end{enumerate}
  
  \item
  Let $\lambda \in [-2,0)$.
  Going through the cases (i), (ii), (iii), and (iv) from \cref{theorem:foscolo-harmonic-forms-on-cone}, we will find that there are six linearly independent harmonic homogeneous $2$-forms of order $-2$ in case (iii), but no other harmonic homogeneous forms.
  Using the notation from the theorem, we get the following:
  
  \begin{enumerate}[label=(\roman*)]
   \item
   In this case, we get a $1$-form that is an eigenform of the Laplacian on $\SO(3)$ for the eigenvalue $\lambda(\lambda + 2)<0$, which must be $0$ by \cref{theorem:spectrum-of-laplacian-on-so3}.
   
   \item
   In this case, we get a $1$-form that is an eigenform of the Laplacian on $\SO(3)$ for the eigenvalue $(\lambda + 2)^2<4$, which must be $0$ by \cref{theorem:spectrum-of-laplacian-on-so3}.
   
   \item
   In this case, we get a $1$-form that is an eigenform of the Laplacian on $\SO(3)$ for the eigenvalue $\lambda^2$.
   There are six of these by \cref{theorem:spectrum-of-laplacian-on-so3} for $\lambda = -2$ and none for $\lambda \in (-2, 0)$.
   In the case of $\lambda = -2$ all six eigenforms give rise to harmonic $2$-forms of order $\lambda = -2$ on $C(\SO(3))$.
   
   \item
   In this case, we get a $2$-form $\beta_4$ that is an eigenform of the Laplacian on $\SO(3)$ for the eigenvalue $(\lambda + 2)^2<4$.
   The Hodge dual $* \beta_4$ is then a $1$-form that is an eigenform for the same eigenvalue, which must be $0$ by \cref{theorem:spectrum-of-laplacian-on-so3}.
   \qedhere
  \end{enumerate}
 \end{enumerate}
\end{proof}

For an application later we will not only need to know how many harmonic homogeneous forms there are, but also how many harmonic homogeneous forms \emph{with $\log(r)$ coefficients} there are.
Often, these two notions coincide, and the following proposition asserts that this is also the case in our setting.

\begin{definition}
 Let $\Sigma$ be a connected Riemannian manifold and $C=C(\Sigma)$ its cone.
 For $\lambda \in \R$, define
 \[
  \mathcal{K}(\lambda)_{\Delta_{p,C(\Sigma)}}
  =
  \left\{
   \begin{matrix}
    \text{
     $\gamma = \sum_{j=0}^m (\log r)^j \gamma_j$ for $m \geq 0$, $\gamma_j \in \Omega^p(C(\Sigma))$, 
     such that
    }
    \\
    \text{
     $\Delta_{p,C(\Sigma)} \gamma=0$, where each $\gamma_j$ is homogeneous of order $\lambda$
    }
   \end{matrix}
  \right\}.
 \]
\end{definition}

\begin{proposition}
\label{proposition:critical-dimension}
 Assume that $\dim C=4$.
 Let $\gamma = \sum_{j=0}^m (\log r)^j \gamma_j \in \mathcal{K}(-2)_{\Delta_{2,C(\Sigma)}}$, then $\gamma_j=0$ for $j>0$.
\end{proposition}

\begin{proof}
 By \cite[Proposition A.6]{Foscolo2020} we have $m \leq 1$ and $\Delta_{2,C(\Sigma)} \gamma_0=0=\Delta_{2,C(\Sigma)} \gamma_1$.
 To prove the claim, it suffices to show that $\gamma_1=0$.
 Write $\gamma_1 = \left( \frac{dr}{r}\wedge\alpha + \beta \right)$.
 By \cref{lemma:laplacian-on-cone-explicit}:
 \begin{align}
     \Delta_{2,C(\Sigma)} \gamma_1&=r^{-2}
     \left(
     \frac{\d r}{r} \wedge A+B)
     \right),
     \quad
     \text{ where}
     \\
     \begin{split}
     A&=
     \log (r)
     \left( \Delta_{1,\Sigma} \alpha-2\d^* \beta \right)
     +2 \alpha,
     \\
     B&=
     \log (r)
     \left( \Delta_{2,\Sigma} \beta - 2 \d \alpha \right)
     +2 \beta.
     \end{split}
 \end{align}
 Thus, comparing the degree $0$ coefficient of the polynomials in $\log(r)$ in the equation $0=\Delta \gamma_1$ immediately gives $\alpha=0=\beta$.
\end{proof}

\subsection{Harmonic Forms on Eguchi-Hanson Space}
\label{subsubsection:harmon-forms-on-eh-space}

In the previous section we looked at certain harmonic forms on $(\C^2 \setminus \{0\})/\{\pm 1 \}$.
The Eguchi-Hanson space $\XEH$ is asymptotic to the cone $(\C^2 \setminus \{0\})/\{\pm 1 \}$, and we can say a great deal about harmonic forms on $\XEH$ just from knowing the harmonic forms on $(\C^2 \setminus \{0\})/\{\pm 1 \}$.
This is a consequence of the work of Lockhart and McOwen (cf. \cite{Lockhart1985,Lockhart1987}) and will be the content of this section.

We will want statements about harmonic forms in certain weighted Hölder spaces.
These spaces are defined in the following:

\begin{definition}
\label{definition:weighted-hoelder-norms}
 Define the weight functions
 \begin{align*}
  w: \XEH & \rightarrow \R _{\geq 0}
  &
  w: \XEH \times \XEH & \rightarrow \R _{\geq 0}
  \\
  x & \mapsto 
  1+|\rho(x)|,
  &
  (x,y) & \mapsto 
  \min
  \{
   w(x),w(y)
  \}
  .
 \end{align*}
 Here, $\rho : \XEH \rightarrow \C^2/\{\pm 1\}$ is the blowup map explained in \cref{lemma:eguchi-hanson-blowup-map}.
 Let $U \subset \XEH$.
 For $\alpha \in (0,1)$, $\beta \in \R$, $k \in \N$, and $f \in \Omega^p(\XEH)$ define the \emph{weighted Hölder norm of $f$} via
 \begin{align*}
  \left[
   f
  \right]
  _{C^{0,\alpha}_{\beta}(U)}
  &:=
  \sup
  _{
   \substack{
   x,y \in U \\
   d_{g_{(1)}} (x,y) \leq w(x,y)
   }
  }
  w(x,y)^{\alpha-\beta}
  \frac{\left| f(x)-f(y) \right|_{g_{(1)}}}{d_{g_{(1)}}(x,y)^\alpha},
  \\
  \|{
   f
  }_{L^\infty_{\beta}(U)}
  &:=
  \|{
   w^{-\beta}
   f
  }_{L^\infty(U)},
  \\
  \|{
   f
  }_{C^{k,\alpha}_{\beta}(U)}
  &:=
  \sum_{j=0}^k
  \|{
   \nabla^j f
  }_{L^\infty_{\beta-j}(U)}
  +
  \left[
   \nabla^j f
  \right]_{C^{0,\alpha}_{\beta-j}(U)}
 \end{align*}
 In these definitions, all vector norms are computed using the metric $g_{(1)}$ on $\XEH$, and the appearing connection is the Levi-Civita connection of this metric.
 The term $f(x)-f(y)$ in the first line denotes the difference between $f(x)$ and the parallel transport of $f(y)$ to the fibre $\Omega^p(\XEH)|_{x}$ along one of the shortest geodesics connecting $x$ and $y$.
 When $U$ is not specified, take $U=\XEH$.
\end{definition}

Sobolev norms with these weight functions were introduced in \cite{Lockhart1985,Lockhart1987}.
The use of the corresponding Hölder norms can be traced back at least to \cite[Section 9]{Lee1987}.
Throughout the article we will set $\beta$ to be a negative number.
Informally, an element in the $C^{k,\alpha}_{\beta}$ Hölder space decays like $d_{g_{(1)}}(\cdot, \rho^{-1}(0))^\beta$, as $d_{g_{(1)}}(\cdot, \rho^{-1}(0)) \rightarrow \infty$.
Even for the choice of $\beta=0$ these norms differ from ordinary Hölder norms, because a weighting is applied to derivatives.

We will now make the meaning of \emph{$\XEH$ being asymptotic to a cone} precise.

\begin{definition}
\label{definition:asymptotically-conical}
 Let $\Sigma$ be a connected Riemannian manifold and $C=C(\Sigma)$ be its cone with cone metric $g_C$.
 A Riemannian manifold $(M, g_M)$ is called \emph{asymptotically conical with cone $C$ and rate $\mu <0$} if there exists a compact subset $L \subset M$, a number $R>0$, and a diffeomorphism $\phi: (R, \infty) \times \Sigma \rightarrow M \setminus L$ satisfying
 \[
  |
   \nabla^k(\phi^* (g_M)-g_C)
  |_{g_C}
  =
  \mathcal{O}(\varrho^{\mu-k})
  \text{ for all }
  k \geq 0
  \text{ as }
  \varrho \rightarrow \infty.
 \]
 Here, $\nabla$ denotes the Levi-Civita connection with respect to $g_C$ and $\varrho: (0, \infty) \times \Sigma \rightarrow (0, \infty)$ is the projection onto the first component.
\end{definition}

The following is then a consequence of \cref{lemma:eguchi-hanson-ALE-condition}:

\begin{proposition}
 The Eguchi-Hanson space $\XEH$ endowed with the metric $g_{(1)}$ is asymptotically conical with cone $C=C(\SO(3))$ and rate $\mu=-4$.
\end{proposition}

We then have the following results about harmonic forms in $L^2$ on Eguchi-Hanson space:

\begin{lemma}
\leavevmode
\label{lemma:harmonic-forms-on-eguchi-hanson-space}
 \begin{enumerate}
 \item
  We have $H^2_{\text{sing}}(\XEH)=H^2_{\text{deRham}}(\XEH)=\R$.
  Define $\nu \in \Omega^2(\XEH)$ to be
  \begin{align}
  \label{equation:nu}
   \nu
   :=
   f_1(r)^{-6}r
   \d r \wedge \eta^1
   -
   f_1(r)^{-2}
   \eta^2 \wedge \eta^3
  \end{align}  
  and endow $\XEH$ with the metric $g_{(1)}$.
  Then $\nu \in L^2(\Lambda^2(\XEH))$, $\Delta_{g_{(1)}} \nu=0$, $[\nu]$ generates $H^2_{\text{deRham}}(\XEH)$, and $\nu$ is the unique element in $L^2(\Lambda^2(\XEH)) \cap [\nu]$ satisfying $\Delta_{g_{(1)}} \nu=0$.
  Moreover, $\nu \in C^{2,\alpha}_{-4}(\Lambda^2(\XEH))$.
  Away from the exceptional orbit $\rho^{-1}(0) \simeq S^2$, we have that
  \begin{align*}
   \nu
   =
   \d \theta
   \text{, where }
   \theta
   =
   - f_1(r)^{-2} \eta^1.
  \end{align*}
  
  \item
  The $L^2$-kernels of $\Delta_{g_{(1)}}$ acting on forms of different degrees are as follows:
  \begin{align*}
   \Ker (\Delta_{g_{(1)}}:L^2(\Lambda^2(\XEH))\rightarrow L^2(\Lambda^2(\XEH)) )
   & = \< \nu \>,
   \\
   \Ker (\Delta_{g_{(1)}}:L^2(\Lambda^p(\XEH))\rightarrow L^2(\Lambda^p(\XEH)) )
   & = 0
   \text{ for }
   p \neq 2.
  \end{align*}
  For $\beta \in [-4,-2)$ they coincide with the $C^{2,\alpha}_{\beta}$-kernels.
 \end{enumerate}
\end{lemma}

\begin{proof}
 \leavevmode
 \begin{enumerate}
 \item
 We have that $\XEH=T^* S^2$ as smooth manifolds, therefore $H^2_{\text{sing}}(\XEH)=\R$.
 On smooth manifolds $H^2_{\text{sing}}(\XEH)=H^2_{\text{deRham}}(\XEH)$ by de Rham's Theorem.
 
 One checks with a direct computation that $\nu$ from \cref{equation:nu} is closed and anti-self-dual, and therefore co-closed.
 The equality $\nu = \d \theta$ follows from a direct computation as well.
  
 One checks through direct calculation that $\nu \in C^{2,\alpha}_{-4}(\Lambda^2(\XEH))$.
 Furthermore, 
 $
  C^{2,\alpha}_{-4} \subset
  L^\infty_{-4} \subset
  L^2,
 $
 so $\nu$ is an element in $L^2(\Lambda^2(\XEH))$.
 
  By Poincaré duality, we have $H^2_{\text{cs}}(\XEH)=H^2_{\text{sing}}(\XEH)=\R$, where $H^2_{\text{cs}}(\XEH)$ denotes the de Rham cohomology with compact support.
  \cite[Example (0.15)]{Lockhart1987} and \cite[Theorem (7.9)]{Lockhart1987} give that the map
  \begin{align*}
   \begin{split}
    \mathcal{H}^2(\XEH)
    :=
    \{
     \xi \in L^2(\Lambda^2 T^*\XEH)
     :
     \d \xi = \d ^* \xi =0
    \}
    &\rightarrow 
    \Im
    \left(
    H^2_{\text{cs}}(\XEH)
    \hookrightarrow
    H^2_{\text{deRham}}(\XEH)
    \right)
    \\
    \xi & \mapsto [\xi]
   \end{split}
  \end{align*}
  is an isomorphism.
  Thus $[\nu]$ generates $H^2_{\text{deRham}}(\XEH)$ and $\nu \in [\nu]$ is the unique element in $L^2(\Lambda^2(\XEH)) \cap [\nu]$ satisfying $\d \nu=0$, $\d^* \nu =0$.
    
  It remains to check that $\nu$ is also the unique element in $L^2(\Lambda^2(\XEH)) \cap [\nu]$ satisfying $\Delta_{g_{(1)}} \nu=0$.
  This holds, because the equations $\Delta_{g_{(1)}} \nu=0$ and $(\d + \d^*)\nu=0$ are equivalent by the same integration by parts argument as in the compact case.
  
  \item
  The first line is a restatement of the previous point.
  The other lines are \cite[Example (0.15)]{Lockhart1987} with proof in \cite[Theorem (7.9)]{Lockhart1987}.
  
  The $L^2$-kernels coincide with the $C^{2,\alpha}_{\beta}$-kernels, as $C^{2,\alpha}_{\beta}(\Lambda^p(\XEH))$ embeds into $L^2(\Lambda^p(\XEH))$ for $\beta< -2$ and the explicit description of the $L^2$-kernels shows that all kernel elements are actually in $C^{2,\alpha}_{\beta}(\Lambda^p(\XEH))$ for $\beta\geq -4$.
  \qedhere
 \end{enumerate}
\end{proof}

\begin{remark}
 Note that $\nu$ from the lemma cannot have compact support by the unique continuation property for elliptic equations.
 We only have that $[\nu]$ contains a form of compact support.
 For general $k > 0$, we have that 
 $f_k(r)^{-6}r
   \d r \wedge \eta^1
   -
   f_k(r)^{-2}
   \eta^2 \wedge \eta^3$
 is $\Delta_{g_{(k)}}$-harmonic.
\end{remark}

The previous lemma makes statements about the $L^2$-kernels of the Laplacian on $\XEH$ acting on $p$-forms.
Using the results from the previous section about harmonic forms on $\C^2/\{\pm 1\}$, we can rule out additional harmonic forms even in some of the weighted Hölder spaces that do not embed into $L^2$.
The key proposition that will be proved throughout the rest of this section is the following:

\begin{proposition}
\label{proposition:hoelder-kernel-of-laplacian}
 For $\beta \in (-4,0)$, the kernels of the $\Delta_{g_{(1)}}$ acting on forms in $C^{2,\alpha}_{\beta}$ of different degrees are as follows:
  \begin{align*}
   \Ker (\Delta_{g_{(1)}}:C^{2,\alpha}_{\beta}(\Lambda^2(\XEH))
   \rightarrow 
   C^{0,\alpha}_{\beta-2}(\Lambda^2(\XEH)) )
   & = \< \nu \>,
   \\
   \Ker (\Delta_{g_{(1)}}:C^{2,\alpha}_{\beta}(\Lambda^p(\XEH))
   \rightarrow 
   C^{0,\alpha}_{\beta-2}(\Lambda^p(\XEH)) )
   & = 0
   \text{ for }
   p \neq 2.
  \end{align*}
\end{proposition}

The connection between the Laplacian on Eguchi-Hanson space and its cone is described in the following results taken from \cite[Section 4]{Karigiannis2020a} which were developed in \cite{Lockhart1985,Lockhart1987}.
The theory works for a much bigger class of operators, but we will only reproduce it for the Laplacian here.
It turns out that the main work in proving this proposition is showing that there are no harmonic $2$-forms on $\XEH$ asymptotic to a non-zero element of the six-dimensional space of harmonic forms of rate $-2$ on the cone from \cref{proposition:harmonic-on-eguchi-cone}.

\begin{definition}
\label{definition:weighted-sobolev-norms}
 Let $M$ be asymptotically conical and let the notation be as in \cref{definition:asymptotically-conical}.
 Denote by $\varrho:C(\Sigma) \rightarrow \R_{\geq 0}$ the radius function, and use the same symbol to denote a map from $M$ to $\R_{> 0}$ that agrees with $\phi _* \varrho$ on $\phi(R, \infty) \subset M$.
 Let $E$ be a vector bundle with metric and metric connection $\nabla$ over $M$.
 Then, for $1>p>\infty$, $l \geq 0$, $\lambda \in \R$ denote by $L^p_{l,\lambda}$ the completion of $C^\infty_{\text{cs}}(E)$ with respect to the norm
 \begin{align*}
  \|{ \gamma }_{L^p_{l,\lambda}}
  =
  \left(
   \sum_{j=0}^l
   \int_M
   |
    \varrho^{-\lambda+j}
    \nabla^j \gamma
   |^p
   \varrho^{-4}
   \vol_M
  \right)^{1/p}.
 \end{align*}
 The space $L^p_{l,\lambda}$ is called the \emph{$L^p$-Sobolev space with $l$ derivatives and decay faster than $\lambda$}.
\end{definition}

\begin{theorem}[Theorem 4.10 in \cite{Karigiannis2020a}]
\label{theorem:kernel-change-at-critical-rate}
 For $\lambda \in \R$, denote by $\Delta_{p,g_{(1)}}: L^q_{2,\lambda}(\Lambda^p(\XEH))
 \rightarrow
 L^q_{0,\lambda-2}(\Lambda^p(\XEH))$
 the Laplacian of the metric $g_{(1)}$ acting on $p$-forms.
 Then, $\Ker \Delta_{p,g_{(1)}}$ is invariant under changes of $\lambda$, as long as we do not hit any critical rates.
 That is, if the interval $[\lambda, \lambda']$ is contained in the complement of $\mathcal{D}_{\Delta_{p,(\C^2 \setminus \{ 0 \})/\{\pm 1\}} }$, then
 \begin{align*}
  &
  \Ker \left( 
  \Delta_{p,g_{(1)}}: L^q_{2,\lambda}(\Lambda^p(\XEH))
 \rightarrow
 L^q_{0,\lambda-2}(\Lambda^p(\XEH))
  \right)
  \\
  =&
  \Ker
  \left( 
  \Delta_{p,g_{(1)}}: L^q_{2,\lambda'}(\Lambda^p(\XEH))
 \rightarrow
 L^q_{0,\lambda'-2}(\Lambda^p(\XEH))
  \right).
 \end{align*}
\end{theorem}

\begin{proposition}[Theorem 4.20 in \cite{Karigiannis2020a}]
\label{proposition:index-change-at-critical-rate}
 Let $\lambda_1 < \lambda_2$ such that $\mathcal{K}(\lambda_i)_{\Delta_{p,C(\Sigma)}}=0$ for $i \in \{1,2\}$.
 Then, the maps
 \begin{align*}
  \Delta_{p,g_{(1)},L^2_{l+2,\lambda_1}}&: 
  L^2_{l+2,\lambda_1}(\Lambda^p(\XEH))
  \rightarrow
  L^2_{l,\lambda_1-2}(\Lambda^p(\XEH))
  \\
  \text{and }
  \Delta_{p,g_{(1)},L^2_{l+2,\lambda_2}}&: 
  L^2_{l+2,\lambda_2}(\Lambda^p(\XEH))
  \rightarrow
  L^2_{l,\lambda_2-2}(\Lambda^p(\XEH))
 \end{align*}
 are Fredholm and the difference in their indices is given by
 \begin{align}
  \label{equation:index-change-formula}
  \ind \left( \Delta_{p,g_{(1)},L^2_{l+2,\lambda_2}} \right) -
  \ind \left( \Delta_{p,g_{(1)},L^2_{l+2,\lambda_1}} \right)
  =
  \sum_{\lambda \in \mathcal{D}_{\Delta_{(\mathbb{C}^2\setminus \{0\}) / \{ \pm 1 \} }} \cap (\lambda_1,\lambda_2)}
  \dim
  \mathcal{K}(\lambda)_{\Delta_{p,(\mathbb{C}^2\setminus \{0\}) / \{ \pm 1 \}}}
 \end{align}
\end{proposition}

Combining everything, we get the following characterisation of harmonic forms with decay:

\begin{theorem}
\label{corollary:kernel-for-eguchi-laplacian}
  For $\lambda \in (-4,0)$, the $L^2_{2,\lambda}$-kernels of $\Delta_{p,g_{(1)}}$ acting on $p$-forms of different degrees are the same as the $L^2$-kernels, namely:
  \begin{align*}
   \Ker (\Delta_{g_{(1)}}:L^2_{2,\lambda}(\Lambda^2(\XEH))\rightarrow L^2_{0,\lambda-2}(\Lambda^2(\XEH)) )
   & = \< \nu \>,
   \\
   \Ker (\Delta_{g_{(1)}}:L^2_{2,\lambda}(\Lambda^p(\XEH))\rightarrow L^2_{0,\lambda-2}(\Lambda^p(\XEH)) )
   & = 0
   \text{ for }
   p \neq 2.
  \end{align*}
\end{theorem}

\begin{proof}
 $0$-forms and $4$-forms:
 it follows from the maximum principle that every harmonic function that decays at infinity must vanish.
 The Hodge star is an isomorphism between $0$-forms and $4$-forms that commutes with the Laplacian, so the statement for $0$-forms implies that statement for $4$-forms.
 
 $1$-forms and $3$-forms:
 the kernel of the Laplacian is zero for rate $-2$ by the second point of \cref{lemma:harmonic-forms-on-eguchi-hanson-space}.
 By the first point of \cref{proposition:harmonic-on-eguchi-cone}, there are no critical rates in the interval $[-2,0)$.
 So, \cref{theorem:kernel-change-at-critical-rate} implies the claim for $1$-forms.
 As above, we get the statement for $3$-forms by using the Hodge star.
 
 $2$-forms:
 by \cref{proposition:harmonic-on-eguchi-cone} the only critical rate in $[-2,0)$ is $-2$.
 The kernel of the Laplacian on $2$-forms stays the same for rates $\lambda \in (-4,-2)$ by \cref{lemma:harmonic-forms-on-eguchi-hanson-space}.
 By \cref{theorem:kernel-change-at-critical-rate}, the dimension of the kernel of the Laplacian acting on $2$-forms with decay $\lambda \in (-4,0)$ may therefore only change at $\lambda = -2$.
 We know from \cref{proposition:critical-dimension,proposition:index-change-at-critical-rate} that the index increases by six when crossing the critical rate $\lambda = -2$.
 We will now check that the dimension of the cokernel decreases by $6$, which implies that the dimension of the kernel does not change.
 
 The dual space of $L^2_{0,\lambda}$ is $L^2_{0,-4-\lambda}$.
 Therefore, the cokernel of 
 $\Delta_{g_{(1)}}:L^2_{2,-2}(\Lambda^2(\XEH))\rightarrow L^2_{0,-4}(\Lambda^2(\XEH))$
 is isomorphic to the kernel of the adjoint operator
 $\Delta_{g_{(1)}}^*=\Delta_{g_{(1)}}:L^2_{2,0}(\Lambda^2(\XEH))\rightarrow L^2_{0,-2}(\Lambda^2(\XEH))$.
 Here we used that elements in the cokernel of $\Delta_{g_{(k)}}$ are smooth by elliptic regularity, so it does not matter how many derivatives we demand for sections acted on by the adjoint operator.

 We now explicitly write down six linearly independent harmonic forms in $L^2_{2,0}(\Lambda^2(\XEH))$:
 three of them are the (self-dual) Kähler forms $\omega_1^{(1)}$, $\omega_2^{(1)}$, and $\omega_3^{(1)}$ defined in \cref{lemma:the-eguchi-hanson-metric}.
 
 Analogously, we can define three harmonic \emph{anti-self-dual} forms with respect to $g_{(k)}$ for all $k>0$.
 To this end, extend $\eta^1, \eta^2, \eta^3 \in \so(3)$ from \cref{lemma:the-eguchi-hanson-metric} to \emph{right}-invariant forms on $\SO(3)$, denoted by $\hat{\eta}_1$, $\hat{\eta}_2$, $\hat{\eta}_3$.
 These forms satisfy $\d \hat{\eta}_1 = - \hat{\eta}^2 \wedge \hat{\eta}^3$ etc.
 In analogy to \cref{lemma:the-eguchi-hanson-metric}, define
 \begin{align*}
  \hat{e}^1(r)&=r f_k^{-1}(r) \hat{\eta}^1,
  &
  \hat{e}^2(r)&=f_k(r) \hat{\eta}^2,
  &
  \hat{e}^3(r)&=f_k(r) \hat{\eta}^3
 \end{align*}
 and
 \begin{align*}
  \hat{\omega}_1^{(k)}&=\d t \wedge \hat{e}^1-\hat{e}^2 \wedge \hat{e}^3, &
  \hat{\omega}_2^{(k)}&=\d t \wedge \hat{e}^2-\hat{e}^3 \wedge \hat{e}^1, &
  \hat{\omega}_3^{(k)}&=\d t \wedge \hat{e}^3-\hat{e}^1 \wedge \hat{e}^2.
 \end{align*}
 One checks through computation that $\hat{\omega}_i^{(k)}$ are closed and anti-self-dual, and therefore harmonic.
 A priori, they are defined on $\R_{>0} \times \SO(3)$, and it remains to check that they extend to all of $\XEH$.
 We have $\hat{\omega}_2^{(k)} = \d \, (r \hat{\eta}^2)$ and $\hat{\omega}_3^{(k)} = \d \, (r \hat{\eta}^3)$, where $r \hat{\eta}^2$ and $r \hat{\eta}^3$ are well-defined $1$-forms on all of $\XEH$, because they vanish as $r \rightarrow 0$.
 Therefore, $\hat{\omega}_2^{(k)}$ and $\hat{\omega}_3^{(k)}$ are well-defined on $\XEH$.
 
 We have that 
 $\hat{\omega}_1^{(k)} = 
   r f^{-2}_k(r) \d r \wedge \hat{\eta}^1
 - f^{-2}_k(r) \hat{\eta}^2 \wedge \hat{\eta}^3$,
 where the first summand vanishes as $r \rightarrow 0$, and the second summand is a multiple of the volume form on $\SO(3) \times_{\SO(2)} \{0 \} \simeq S^2$ pulled back under the projection
 \begin{align*}
  \SO(3) \times_{\SO(2)} V & \rightarrow \SO(3) \times_{\SO(2)} V
  \\
  (g, x) & \mapsto (g, 0).
 \end{align*}
 Thus $\hat{\omega}_1^{(k)}$ is also defined on all of $\XEH$.
 The forms $\eta^1,\eta^2,\eta^3,\hat{\eta}^1,\hat{\eta}^2,\hat{\eta}^3$ are linearly independent which implies that $\omega_1^{(k)},\omega_2^{(k)},\omega_3^{(k)},\hat{\omega}_1^{(k)},\hat{\omega}_2^{(k)},\hat{\omega}_3^{(k)}$ are linearly independent.
 
 Last, note that for each $g \in \SO(3)$ we can express $\hat{\eta}^i(g)$ as a linear combination of $\eta^i(g)$.
 Each $\eta^i$ decays like $r^{1/2}$ as $r \rightarrow \infty$, which shows that the $\hat{\omega}_i^{(k)}$ have the same decay as the Hyperkähler triple $\omega_i^{(k)}$, which is covariant constant.
 Thus, we have that $\omega_i^{(1)}, \hat{\omega}_i^{(1)} \in L^2_{2,0}(\Lambda^2(\XEH))$, but $\notin L^2_{2,-\epsilon}(\Lambda^2(\XEH))$ for all $\epsilon > 0$ and $i \in \{1,2,3\}$.
 
 Therefore, the dimension of the cokernel of $\Delta_{g_{(1)}}:L^2_{2,\lambda}(\Lambda^2(\XEH))\rightarrow L^2_{0,\lambda-2}(\Lambda^2(\XEH))$ changes by six when crossing the critical rate $\lambda=-2$, and the dimension of the kernel stays the same.
\end{proof}

The claim for $1$-forms in \cref{corollary:kernel-for-eguchi-laplacian} can also be seen as follows:
if $a \in L^q_{2,\lambda}(\Lambda^1(\XEH))$ such that $\Delta_{g(1)}=0$, then
\[
    \Delta_{g(1)} |a|^2
    =
    - | \nabla a|^2+ \< \nabla^* \nabla a,a \>
    =
    - | \nabla a|^2
    \leq
    0,
\]
where we used \cite[Equation 6.18]{Freed1991} in the first step and used the Weitzenböck formula on $1$-forms and Ricci-flatness in the second step.
By the maximum principle \cite[Theorem 2.2]{Gilbarg2001} together with the fact that $|a|$ decays at infinity, we have that $a=0$.

\Cref{proposition:hoelder-kernel-of-laplacian} is now an immediate consequence of \cref{corollary:kernel-for-eguchi-laplacian}.

\begin{proof}[Proof of \cref{proposition:hoelder-kernel-of-laplacian}]
 For $\epsilon > 0$ we have that $C^{2,\alpha}_{\beta-\epsilon}$ is embedded in $L^2_{2,\beta}$, so the claim follows from \cref{corollary:kernel-for-eguchi-laplacian}.
\end{proof}

\section{Torsion-Free $G_2$-Structures on the Generalised Kummer Construction}
\label{section:torsion-free-structures-on-the-generalised-kummer-construction}

In the two articles \cite{Joyce1996}, Joyce constructed the first examples of manifolds with holonomy equal to $G_2$.
One starts with the flat $7$-torus, which admits a flat $G_2$-structure.
A quotient of the torus by maps preserving the $G_2$-structure still carries a flat $G_2$-structure, but has \emph{singularities}.
The maps are carefully chosen, so that the singularities are modelled on $T^3 \times \C^2/\{\pm 1\}$.
By the results of \cref{section:analysis-on-eguchi-hanson}, $T^3 \times \C^2/\{\pm 1\}$ has a family of resolutions $T^3 \times \XEH \rightarrow T^3 \times \C^2/\{\pm 1\}$ depending on one real parameter, where $\XEH$ denotes the Eguchi-Hanson space, and the parameter defines the size of a minimal sphere in $\XEH$.
We can define a smooth manifold by gluing these resolutions over the singularities in the quotient of the torus.

The product manifold $T^3 \times \XEH$ carries the product $G_2$-structure from \cref{equation:product-g2-structure}.
That means we have two torsion-free $G_2$-structures on our glued manifold:
one coming from flat $T^7$, and the product $G_2$-structure near the resolution of the singularities.
We will interpolate between the two to get one globally defined $G_2$-structure.
This will no longer be torsion-free, but it will have small enough torsion in the sense of \cref{theorem:original-torsion-free-existence-theorem}.
This is the argument that was used in \cite{Joyce1996} to prove the existence of a torsion-free $G_2$-structure, and the construction of this $G_2$-structure with small torsion is the content of \cref{subsection:resolutions-of-t7-gamma}.

\Cref{subsection:the-laplacian-on-r3-times-X,subsection:the-laplacian-on-M,subsection:the-existence-theorem} give an alternative proof of the existence of a torsion-free $G_2$-structure on this glued manifold.

\subsection{Resolutions of $T^7/\Gamma$}
\label{subsection:resolutions-of-t7-gamma}

We briefly review the generalised Kummer construction as explained in \cite{Joyce1996}.
Let $(x_1,\dots,x_7)$ be coordinates on $T^7=\R^7/\Z^7$, where $x_i \in \R/\Z$, endowed with the flat $G_2$-structure $\varphi_0$ from \cref{definition:g2}.
Let $\alpha,\beta,\gamma:T^7 \rightarrow T^7$ defined by
\begin{align}
\label{equation:alpha-beta-gamma}
\begin{split}
 \alpha: (x_1,\dots,x_7)
 &\mapsto
 (-x_1,-x_2,-x_3,-x_4,x_5,x_6,x_7),
 \\
 \beta: (x_1,\dots,x_7)
 &\mapsto
 \left(
  -x_1,\frac{1}{2}-x_2,x_3,x_4,-x_5,-x_6,x_7
 \right),
 \\
 \gamma: (x_1,\dots,x_7)
 &\mapsto
 \left(
  \frac{1}{2}-x_1,x_2,\frac{1}{2}-x_3,x_4,-x_5,x_6,-x_7
 \right).
 \end{split}
\end{align}

Denote $\Gamma:= \< \alpha,\beta,\gamma \>$.
The next lemmata collect some information about the orbifold $T^7/\Gamma$:

\begin{lemma}[Section 2.1 in part I, \cite{Joyce1996}]
 $\alpha,\beta,\gamma$ preserve $\varphi_0$, we have $\alpha^2=\beta^2=\gamma^2=1$, and $\alpha, \beta,\gamma$ commute.
 We have that $\Gamma \simeq \Z^3_2$.
\end{lemma}

\begin{lemma}[Lemma 2.1.1 in part I, \cite{Joyce1996}]
 The elements $\beta \gamma$, $\gamma \alpha$, $\alpha \beta$, and $\alpha \beta \gamma$ of $\Gamma$ have no fixed points on $T^7$.
 The fixed points of $\alpha$ in $T^7$ are $16$ copies of $T^3$, and the group $\< \beta , \gamma\>$ acts freely on the set of $16$ $3$-tori fixed by $\alpha$.
 Similarly, the fixed points of $\beta$, $\gamma$ in $T^7$ are each $16$ copies of $T^3$, and the groups $\< \alpha, \gamma \>$ and $\< \alpha, \beta \>$ act freely on the sets of $16$ $3$-tori fixed by $\beta, \gamma$ respectively.
\end{lemma}

\begin{lemma}[Lemma 2.1.2 in part I, \cite{Joyce1996}]
\label{lemma:tubular-neighbourhood-lemma}
 The singular set $L$ of $T^7/\Gamma$ is a disjoint union of $12$ copies of $T^3$.
 There is an open subset $U$ of $T^7/\Gamma$ containing $L$, such that each of the $12$ connected components of $U$ is isometric to $T^3 \times \left(B^4_\zeta /\{ \pm 1\} \right)$, where $B^4_\zeta$ is the open ball of radius $\zeta$ in $\R^4$ for some positive constant $\zeta$ ($\zeta=1/9$ will do).
\end{lemma}

For $0<t \ll 1$ we now define a compact $7$-manifold $N_t$, which can be thought of as a resolution of the orbifold $T^7 /\Gamma$, and a one-parameter family of closed $G_2$-structures $\varphi^t$ thereon.
We can choose an identification $U \simeq L \times \left(B^4_\zeta /\{ \pm 1\} \right)$ such that we can write on $U$
\begin{align*}
 \varphi_0
 &=
 \delta_1 \wedge \delta_2 \wedge \delta_3-
 \sum_{i=1}^3
 \omega_i \wedge \delta_i,
 &
 *\varphi_0
 &=
 \frac{1}{2} \omega_1 \wedge \omega_1-
 \sum
 _
 {\substack{(i,j,k) = (1,2,3)\\
 \text{and cyclic permutation}
 }}
 \omega_i \wedge \delta_j \wedge \delta_k,
\end{align*}
where $\delta_1,\delta_2,\delta_3$ are covariant constant orthonormal $1$-forms on $L$, and $\omega_1,\omega_2,\omega_3$ are the Hyperkähler triple from \cref{section:g2-structures}.

As before, denote by $\XEH$ the Eguchi-Hanson space and by $\rho: \XEH \rightarrow \C^2/\{ \pm 1\}$ the blowup map from \cref{lemma:eguchi-hanson-blowup-map}.
Define $\check{r} := | \rho |: \XEH \rightarrow \R_{\geq 0}$.
For $t \in (0,1)$, let $\hat{U}:=\hat{U}_t:=L \times \{x \in \XEH: \check{r}(x) < \zeta t^{-1} \}$.
Define
\begin{align}
\label{equation:resolution-of-t7-gamma}
 N_t:=
 \left(
 (T^7/\Gamma) \setminus L \sqcup \hat{U}
 \right)/\sim,
\end{align}
where for $x=(x_h, x_v) \in U \subset L \times \C^2/\{ \pm 1\}$ and $y=(y_h, y_v) \in \hat{U} \subset L \times \XEH$ we have $x \sim y$ if $x_h=y_h$ and $t \cdot \rho(y_v)=x_v$. 
The smooth manifold $N_t$ also comes with a natural projection map $\pi: N_t \rightarrow T^7/\Gamma$ induced by $\rho$,
and we extend $\check{r}$ to a map on all of $N_t$ via
\begin{align*}
 \check{r}: N_t & \rightarrow \R_{\geq 0}
 \\
 x & \mapsto
 \begin{cases}
  | \rho(x) | & \text{ if } x \in \hat{U}, \\
  t^{-1}\zeta & \text{ otherwise}.
 \end{cases}
\end{align*}

Write 
\begin{align}
    \label{equation:r_t}
    r_t := t\check{r}
\end{align}
and choose a non-decreasing function
\begin{align}
    \label{equation:cut-off}
    \chi:[0,\zeta] \rightarrow [0,1]
    \text{ such that }
    \chi(s)=0 \text{ for } s\leq \zeta/4 \text{ and } \chi(s)=1 \text{ for } s \geq \zeta/2
\end{align}
and set
\begin{align}
\label{equation:tilde-omega}
 \tilde{\omega}_{i}
 :=
 \omega_i^{(1)} -
 \d
 \left(
  \chi(r_t)
  \tau_i^{(1)}
 \right).
\end{align}
The $\tau_i^{(1)}$ were defined in \cref{lemma:eguchi-hanson-ALE-condition}, and are the difference between the flat Hyperkähler triple on $\C^2 /\{ \pm 1 \}$ and the Hyperkähler triple $(\omega_1^{(1)}, \omega_2^{(1)}, \omega_3^{(1)})$ on $\XEH$.
On $\hat{U}$ we have $\tilde{\omega}_{i}=\omega_i$ where $r_t >\zeta/2$, and $\tilde{\omega}_{i}=\omega_i^{(1)}$ where $r_t < \zeta/4$.
Now define a $3$-form $\varphi^t \in \Omega^3(N_t)$ and a $4$-form $\vartheta^t \in \Omega^4(N_t)$ as follows:
on $(T^7/\Gamma) \setminus U \subset N_t$, set $\varphi^t=\varphi$ and $\vartheta^t=*\varphi$.
On $\hat{U} \subset L \times \XEH$ let
\begin{align}
 \label{equation:g2-structure-on-kummer-construction}
 \varphi^t
 &:=
 \delta_1 \wedge \delta_2 \wedge \delta_3-
 t^2
 \sum_{i=1}^3
 \tilde{\omega}_{i} \wedge \delta_i,
 \\
 \vartheta^t
 &:=
 t^4
 \frac{1}{2} \tilde{\omega}_{1} \wedge \tilde{\omega}_{1}-
 t^2
 \sum
 _
 {\substack{(i,j,k) = (1,2,3)\\
 \text{and cyclic permutation}
 }}
 \tilde{\omega}_{i} \wedge \delta_j \wedge \delta_k.
\end{align}

This definition mimics the product situation explained in \cref{section:g2-structures}. 
For small $t$, the $3$-form $\varphi^t$ is a $G_2$-structure and therefore induces a metric $g^t$.
Both $\varphi^t$ and $\vartheta^t$ are closed forms, so, if $*\varphi^t=\vartheta^t$, then $\varphi^t$ would be a torsion-free $G_2$-structure by \cref{theorem:torsion-free-g2-structures-characterisation}.
However, this does not hold, and $\varphi^t$ is not a torsion-free $G_2$-structure.
The following $3$-form $\psi^t$ is meant to measure the torsion of $\varphi^t$:
\begin{align}
\label{equation:definition-psi-t}
 * \psi^t
 =
 \Theta(\varphi^t)-\vartheta^t.
\end{align}
Its crucial properties are:
\begin{lemma}
\label{lemma:kummer-construction-pregluing-estimate}
 Let $\psi^t \in \Omega^3(N_t)$ as in \cref{equation:definition-psi-t}.
 There exists a positive constant $c$ independent of $t$ such that
 \begin{align*}
  \d^* \psi^t &= \d^* \varphi^t,
  &
  \|{
   \psi^t
  }_{C^{1,\alpha}}
  &\leq
  ct^4,
 \end{align*}
 where the Hölder norm is defined with respect to the metric $g^t$ and its induced Levi-Civita connection.
\end{lemma}

\begin{proof}
 The equality $\d^* \psi^t = \d^* \varphi^t$ follows from \cref{equation:definition-psi-t} and the fact that $\vartheta^t$ is closed.
 
 The operator $*$ is parallel, so the covariant derivative $\nabla_X$ and $*$ commute for every vector field $X$ on $N_t$, therefore it suffices to estimate $* \psi^t$ rather than $\psi^t$.
 Write
 $\varphi_{\XEH \times L}^{(t)}
 :=
 \delta_1 \wedge \delta_2 \wedge \delta_3-
 t^2
 \sum_{i=1}^3
 \omega_i^{(1)} \wedge \delta_i$
 for the product $G_2$-structure on $\XEH \times L$ and denote the induced metric, which is the product metric, by $g_{\XEH \times L}^{(t)}$.
 By definition of $\varphi^t$ we have $\varphi^t=\varphi_{\XEH \times L}^{(t)}$ on the set $\{ x \in N_t : r_t(x) < \zeta/4 \}$.
 Recall the linear map $T$ and the non-linear map $F$ from \cref{proposition:Theta-estimates} satisfying $\Theta (\varphi + \xi) = *\varphi - T(\xi) - F(\xi)$ for a $G_2$-structure $\varphi$ and a small deformation $\xi$.
 Using this notation, we get:
 \begin{align*}
  \Theta(\varphi^t)- \vartheta^t
  &=
   \Theta
   \left(
    \varphi_{\XEH \times L}^{(t)}
    -
    t^2
    \delta_1
    \wedge
    \d
    \left(
    \chi(r_t)
    \tau_1^{(1)}
    \right)
   \right)
  \\ 
  &\;\;\;\;\;\;
  -
  *_{g_{\XEH \times L}^{(t)}}
  \varphi_{\XEH \times L}^{(t)}
  +
  t^2
  \delta_2 \wedge \delta_3 \wedge
  \d
    \left(
    \chi(r_t)
    \tau_1^{(1)}
  \right)
  \\
  &=
  T
  \left(
  	t^2
    \delta_1
    \wedge
    \d
    \left(
    \chi(r_t)
    \tau_1^{(1)}
    \right)
  \right)
  -
  F
  \left(
    -
    t^2
    \delta_1
    \wedge
    \d
    \left(
    \chi(r_t)
    \tau_1^{(1)}
    \right)
  \right)
  \\ 
  &\;\;\;\;\;\;
  +
  t^2
  \delta_2 \wedge \delta_3 \wedge
  \d
    \left(
    \chi(r_t)
    \tau_1^{(1)}
  \right).
 \end{align*}
 Here we used the equality $\omega^{(k)}_1-\omega_1=\d \tau^{(k)}_1$ from \cref{lemma:eguchi-hanson-ALE-condition} in the first step and the definition of $T$ and $F$ in the second step.
 
 Note that $\Theta(\varphi^t)- \vartheta^t$ is supported on $\{ x \in M : (\zeta/4) t^{-1} < \check{r} < (\zeta/2) t^{-1} \}$.
 Therefore, by \cref{equation:ALE-hyperkaehler-decay},
 \begin{align*}
  \left|
  t^2
  \d
    \left(
    \chi(r_t)
    \tau_1^{(1)}
  \right)
  \right|_{t^2 g_{(1)}}
  &\leq
  \left|
  t^2
  \left(
  \d \chi(r_t)
  \right)
  \tau_1^{(1)}
  \right|_{t^2 g_{(1)}}
  +
  \left|
  t^2
  \chi(r_t)
  \d
  \tau_1^{(1)}
  \right|_{t^2 g_{(1)}}
  \\
  &
  \leq
  ct
  \left|
  t
  \tau_1^{(1)}
  \right|_{t^2 g_{(1)}}
  +
  c
  \left|
  t^2\chi(r_t)
  \d \tau_1^{(1)}
  \right|_{t^2 g_{(1)}}
  \\
  &=
  t \mathcal{O}(\check{r}^{-3})+\mathcal{O}(\check{r}^{-4})
  \leq
  ct^4.
 \end{align*}
 Using the estimates for $T$ and $F$ from \cref{proposition:Theta-estimates} we get the claim.
\end{proof}

\subsection{The Laplacian on $\R^3 \times \XEH$}
\label{subsection:the-laplacian-on-r3-times-X}

In the next section we will prove an estimate for the Laplacian on $2$-forms on $N_t$.
We will use a blowup argument to essentially reduce the analysis on $N_t$ to the analysis on $T^7/\Gamma$ and $\R^3 \times \XEH$.
In this section we will cite a general result for uniformly elliptic operators on product manifolds $\R^n \times Y$ from \cite{Walpuski2013}, where $Y$ is a Riemannian manifold, and use this to find that harmonic $2$-forms on $\R^3 \times \XEH$ are wedge products of parallel forms on $\R^3$ and harmonic forms on $\XEH$.

\begin{definition}[Definition 2.75 in  \cite{Walpuski2013}]
 A Riemannian manifold $Y$ is said to be of \emph{bounded geometry} if it is complete, its Riemann curvature tensor is bounded from above and its injectivity radius is bounded from below.
 A vector bundle over $Y$ is said to be of \emph{bounded geometry} if it has trivialisations over balls of fixed radius such that the transition functions and all of their derivatives are uniformly bounded.
 We say that a complete oriented Riemannian manifold $X$ has \emph{subexponential volume growth} if for each $x \in X$ the function $r \mapsto \vol(B_r(x))$ grows subexponentially, i.e., $\vol(B_r(x))= o(\exp(cr))$ as $r \rightarrow \infty$ for every $c>0$.
\end{definition}

\begin{lemma}[Lemma 2.76 in \cite{Walpuski2013}]
\label{lemma:walpuski-constant-in-r-n-directions}
\label{lemma:constant-in-R-n-direction-general}
 Let $E$ be a vector bundle of bounded geometry over a Riemannian manifold $Y$ of bounded geometry and with subexponential volume growth, and suppose that $D: C^\infty(Y,E) \rightarrow C^\infty(Y,E)$ is a uniformly elliptic operator of second order whose coefficients and their first derivatives are uniformly bounded, that is non-negative, i.e., $\< Da, a \> \geq 0$ for all $a \in W^{2,2}(Y,E)$, and formally self-adjoint.
 Let $p: \R^n \times Y \rightarrow Y$ be the projection onto the second component and $a \in C^\infty(\R^n \times Y, p^* E)$ such that
 \begin{align*}
  \left(
   \Delta_{\R^n}+p^*D
  \right)
  a=0
 \end{align*}
 and $\|{a }_{L^\infty}$ is finite.
 Then $a$ is constant in the $\R^n$-direction, that is $a(x,y)=a(y)$.
 Here, $\Delta_{\R^n}$ acts on a section $a \in C^\infty(\R^n \times Y, p^* E)$ by using the identification $C^\infty(\R^n \times Y, p^* E)=C^\infty(\R^n, C^\infty( Y, E))$.
\end{lemma}

\begin{corollary}
\label{lemma:laplacian-independent-of-r-3-direction}
 Let $Y$ be a manifold of bounded geometry and with subexponential volume growth.
 If $a \in \Omega^2(\R^3 \times Y)$ satisfies $\|{ a }_{L^\infty}< \infty$ and
 \begin{align*}
  \Delta_{\R^3 \times Y} \, a =0,
 \end{align*}
 then $a$ is independent of the $\R^3$-direction.
\end{corollary}

\begin{proof}
    Fix a trivialisation $(\d x_1, \d x_2, \d x_3)$ of the pullback of $\Lambda^1(\R^3)$ to $\R^3 \times \XEH$.
    Let $p: \R^3 \times Y \rightarrow Y$ be the projection onto the second component.
    Write $a \in \Omega^2(\R^3 \times Y)$ as
    \[
        a
        =
        a^{(2)}
        +
        \sum_{i=1}^3 \d x_i \wedge a_i^{(1)}
        +
        \sum_{1 \leq j < k \leq 3} \d x_j \wedge \d x_k \cdot a_{jk}^{(0)},
    \]
    where $a^{(2)} \in \Gamma(p^*(\Lambda^2 T^*Y))$ and $a^{(1)}_i \in \Gamma(p^*(\Lambda^1 T^*Y))$ for $i \in \{1,2,3\}$ and $a^{(0)}_{jk} \in \Gamma(p^*(\Lambda^0 T^*Y))$ for $1 \leq j < k \leq 3$.
    Then $\Delta_{g_{\R^3} \oplus g_{(1)}} a=0$ if and only if
    \[
    (\Delta_{\R^3} + p^* \Delta_Y) a^{(m)}_{\bullet}=0 \text{ for } a^{(m)}_{\bullet} \in \{a^{(2)}, a^{(1)}_i,a^{(0)}_{jk} \}_{1 \leq i \leq 3, 1 \leq j < k \leq 3}.
    \]
    \Cref{lemma:walpuski-constant-in-r-n-directions} then gives that $a$ is independent of the $\R^3$-direction.
\end{proof}

\subsection{The Laplacian on $N_t$}
\label{subsection:the-laplacian-on-M}

We now move on to the heart of the argument:
an operator bound for the inverse of the Laplacian on $N_t$.
The Laplacian on $2$-forms has a kernel of dimension $b^2(N_t)$, so we can only expect such a bound for forms which are not in the kernel.
Standard elliptic theory would give an estimate for forms orthogonal to the kernel.
This estimate would depend on the gluing parameter $t$, but we want a \emph{uniform} estimate, i.e. an estimate independent of $t$.
Proving such an estimate is the content of this section.

We first define weighted Hölder norms analogous to the previous sections.
These norms have the following two important properties:
far away from $L$, they are uniformly equivalent to ordinary Hölder norms, and near $L$ they are uniformly equivalent to the weighted Hölder norms on $\R^3 \times \XEH$, after applying a rescaling map.

\begin{definition}
\label{definition:weight-function}
 For $t \in (0,1)$ define the weight functions
 \begin{align}
 \label{equation:weight-function-on-M}
 \begin{split}
  w_t : N_t & \rightarrow \R_{>0}
  \\
  x & \mapsto t+r_t,
 \end{split}
  \\
  \notag
  w_{\R^3 \times \R^4}: \R^3 \times \R^4 & \rightarrow \R_{>0}
  \\
  \notag
  (x,y) & \mapsto \left| y \right|,
  \\
  w_{\R^3 \times \XEH}: \R^3 \times \XEH & \rightarrow \R _{>0}
  \notag
  \\
  x & \mapsto 
  1+\check{r}
  \notag
 \end{align}
 and for $k \in \N$, $\alpha \in (0,1)$, $\beta \in \R$ the weighted Hölder norms $\|{ \cdot }_{C^{k,\alpha}_{\beta;t}}$ on $N_t$ and $\|{ \cdot }_{C^{k,\alpha}_{\beta}}$ on $\R^3 \times \R^4$ and $\R^3 \times \XEH$ using the same formulae as in \cref{definition:weighted-hoelder-norms} but with $w_t$, $w_{\R^3 \times \R^4}$, and $w_{\R^3 \times \XEH}$ in place of $w$.
 Here, $r_t$ was defined in \cref{equation:r_t} and we have that $w_t=t(1+\check{r})$.
\end{definition}

Roughly speaking, two norms can be defined on the set $\hat{U} \subset N_t$ :
the norm $\|{ \cdot }_{C^{k,\alpha}_{\beta}}$ from $\R^3 \times \XEH$ and the norm $\|{ \cdot }_{C^{k,\alpha}_{\beta;t}}$ that is defined on all of $N_t$.
This is not completely precise, because $\hat{U}$ is a product whose first factor is $L$ rather than $\R^3$, but (twelve copies of) $\R^3$ is the universal cover of $L$, so by pulling back one can evaluate tensors defined on $\hat{U}$ in that norm.
The metric on $\hat{U} \subset N_t$ is scaled by a factor of $t$ compared to the product metric on $\R^3 \times \XEH$, and so the Hölder norms on tensors are related by a rescaling as well.
This will be made precise in \cref{lemma:s_t-second-lemma}.

We now define a way to decompose elements $a \in \Omega^2(N_t)$ into a component $\overline{\pi}_t a$ that is proportional to a cut-off of $\nu \in \Omega^2(\XEH)$ from \cref{equation:nu} on every fibre $\{y\} \times \XEH \subset L \times \XEH$, and a remainder, denoted by $\rho_t a$.
The reason for this is the following:
the Laplacian on $\Im \overline{\pi}_t$ is approximately the Laplacian on $L$, and its inverse has operator norm of order $\mathcal{O}(1)$ uniformly in $t$ as a map $C^{2,\alpha}_{\beta;t}(\Lambda^2(N_t)) \rightarrow C^{0,\alpha}_{\beta;t}(\Lambda^2(N_t))$.
For this to be the case the weight for the norms on the domain and the codomain must be the same.
On $\Im \rho_t$, it will turn out that the Laplacian has operator norm of order $\mathcal{O}(1)$ uniformly in $t$ as a map $C^{2,\alpha}_{\beta;t}(\Lambda^2(N_t)) \rightarrow C^{2,\alpha}_{\beta-2;t}(\Lambda^2(N_t))$.
Here the weight changed in the same way as it did on the non-compact asymptotically conical space $\XEH$, cf. \cref{subsubsection:harmon-forms-on-eh-space}.
In order to prove an estimate of the form $\|{a} \leq c \|{\Delta a}$ we will define norms that incorporate these two different scaling behaviours in this section.
The idea is taken from \cite{Walpuski2017}.

Using the cut-off function $\chi:[0,\zeta] \rightarrow [0,1]$ from \cref{equation:cut-off}, we write $\chi_t := \chi(2 r_t)$ as a shorthand.
Define $\pi_t:\Omega^2(N_t) \rightarrow \Omega^0(L)$ via
\begin{align}
\label{equation:def-pi}
(\pi_t a)(y):=
\< a|_{ \{y\} \times \XEH}, (1-\chi_t) \nu' \>_{L^2,t^2 g_{\XEH}}
\text{ for }
y \in L,
\end{align}
where $\nu' \in \Omega^2(\XEH)$ is a multiple of $\nu$ from \cref{equation:nu} satisfying
$\< \chi_t\nu', \chi_t\nu' \>_{L^2,t^2 g_{\XEH}}=1$.
This is equivalent to $\< \chi_t\nu', \chi_t\nu' \>_{L^2,g_{\XEH}}=1$, i.e. in the metric $g_{\XEH}$ rather than $t^2 g_{\XEH}$, because the $L^2$-norm on $2$-forms is a conformal invariant.
Define
\begin{align}
    \label{equation:iota}
    \begin{split}
        \iota_t: \Omega^0(L) &\rightarrow \Omega^2(N_t)
        \\
        f & \mapsto
        \chi_t \cdot p_L^*f \cdot p_{\XEH}^*\nu'
    \end{split} 
\end{align}
where $p_L: L \times \XEH \rightarrow L$, $p_{\XEH}: L \times \XEH \rightarrow \XEH$ are projection maps.
As written, $(\iota_t f)$ is an element in $\Omega^2(L \times \XEH)$, but because $\supp (\iota_t f) \subset \hat{U}$, we can view it as an element in $\Omega^2(N_t)$.
Then
\begin{align}
\pi_t \iota_t f =f
\text{ for all }
f \in \Omega^0(L).
\end{align}
Last, define 
\begin{align}
\label{equation:pi-bar-definition}
    \overline{\pi}_t:=\iota_t \pi_t \text{ as well as }
    \rho_t := 1-\overline{\pi}_t.
\end{align}

We are now ready to define the composite norms which weight the $\overline{\pi}_t$ and $\rho_t$ components differently.

\begin{definition}
\label{definition:x-y-norms}
For $\alpha \in (0,1)$ and $\beta \in (-1,0)$ let
 \begin{align*}
  \|{a}_{\mathfrak{X}_t}
  &:=
  \|{\rho_t a}_{C^{2,\alpha}_{\beta;t}}
  +
  t^{-3/2}
  \|{\pi_t a}_{C^{2,\alpha}},
  \\
  \|{a}_{\mathfrak{Y}_t}
  &:=
  \|{\rho_t a}_{C^{0,\alpha}_{\beta-2;t}}
  +
  t^{-3/2}
  \|{\pi_t a}_{C^{0,\alpha}}.
 \end{align*}
\end{definition}

In the following, we will always assume that $\alpha$ and $\beta$ are close to $0$.
The most restrictive estimate in which this fact is used is \cref{equation:alpha-beta-smallness-condition}.
For concreteness, one may choose $\alpha = 1/16$ and $\beta = - 1/16$.

\begin{definition}[Approximate kernel]
\label{definition:approximate-kernel}
 Let $C_1, \dots, C_{12}$ be the connected components of $\hat{U}$ and let $\chi_{C_i}$ be the characteristic function of the set $C_i$.
 Then define the \emph{approximate kernel of $\Delta$ on $N_t$} to be
 \begin{align*}
 \mathcal{K}
 :=
 \{
  (1-\chi_t) \pi^* a: a \in \Ker \Delta_{T^7/\Gamma}
 \}
 \oplus
 \spann
 \left(
  (1-\chi_t) \cdot p_{\XEH}^* \nu \cdot \chi_{C_i}
 \right)_{i=1,\dots,12}
 ,
 \end{align*}
 where $\pi: N_t \rightarrow T^7/\Gamma$ is the projection map from the previous section.
\end{definition}

With all this notation in place we can state the linear estimate that will be used later on:

\begin{proposition}
\label{proposition:x-y-laplace-estimate-N_t}
There exists $c$ independent of $t$ such that for $t$ small enough we have $\Im \left( \Delta|_{\mathcal{K}^\perp} \right)=\Im(\Delta)$ and for all $a \in \Omega^2(N_t)$, $a \perp \mathcal{K}$
 \begin{align}
 \label{equation:x-y-laplace-estimate}
  \|{a}_{\mathfrak{X}_t}
  \leq
  c
  \|{\Delta a}_{\mathfrak{Y}_t}.
 \end{align}
\end{proposition}

As is often the case in geometric analysis proofs in which an approximate solution is perturbed to a genuine solution, it is the linear estimate which is the most laborious to prove.
The proof of the linear estimate \cref{proposition:x-y-laplace-estimate-N_t} will be given in \cref{section:linear-estimate-proof}.
In the remainder of \cref{section:torsion-free-structures-on-the-generalised-kummer-construction} we complete the perturbation of the approximately torsion-free $G_2$-structure $\varphi^t$ from \cref{equation:g2-structure-on-kummer-construction} to a torsion-free $G_2$-structure, thereby proving our main theorem, \cref{corollary:kummer-construction-simplified-torsion-free-estimate}.

\subsection{The Existence Theorem}
\label{subsection:the-existence-theorem}

We will now prove the theorem which guarantees the existence of a torsion-free $G_2$-structure when starting from a $G_2$-structure with small torsion.

\begin{theorem}
\label{theorem:torsion-free-final-existence-theorem}
Assume there exist $c',c''>0$ such that $\psi^t \in \Omega^3(N_t)$ satisfies $\d ^* \varphi^t=\d ^* \psi^t$ and
\begin{align*}
\|{
\d^*\psi^t
}_{\mathfrak{Y}_t}
&\leq
c't^4,
\\
\|{\psi^t}_{C^{0,\alpha}_{0;t}}
&\leq
c''t^4.
\end{align*}
Then, for small $t$, there exists $\eta^t \in \Omega^2(N_t)$ such that $\varphi^t+\d \eta$ is a torsion-free $G_2$-structure and $\|{\rho_t \eta}_{C^{2,\alpha/2}_{\beta;t}}+t^{-3/2}\|{\pi_t \eta}_{C^{2,\alpha/2}} \leq ct^4$.
\end{theorem}

To ease notation, we write $\varphi=\varphi^t$, $\psi=\psi^t$, and $\eta=\eta^t$ throughout the proof.

\begin{proof}
We will construct $\eta \in \Omega^2(N_t)$ satisfying
\begin{align}
\label{equation:joyce-slice-theorem-torsion-free-pde}
\Delta \eta
=
\d^* \psi
+
\d ^* (f \psi)
+
*\d \left( F(\d \eta) \right),
\text{ where }
f=\frac{7}{3} \< \varphi, \d \eta\>
\end{align}
and $F$ refers to the map from \cref{proposition:Theta-estimates}.
Set $\eta_0=0$ and, if $\eta_{j-1} \in \Omega^2(N_t)$ is given, let $\eta_j \in \Omega^2(N_t)$ be such that
\begin{align*}
\Delta \eta_j
=
\d^* \psi
+
\d ^* (f_{j-1} \psi)
+
*\d \left( F(\d \eta_{j-1}) \right),
\text{ where }
f_{j-1}
=
\frac{7}{3} \< \varphi, \d \eta_{j-1}\>,
\end{align*}
and such that $\eta_j \perp \mathcal{K}$.
This is well-defined, i.e. such $\eta_j$ exists, because $\Im \d^* \subset \Im \Delta$ and restricting $\Delta$ to $\mathcal{K}^\perp$ does not change its image by \cref{proposition:approximate-kernel-not-too-big}.
We aim to show by induction that $\|{ \eta_j }_{\mathfrak{X}_t} \leq ct^4$.
For $j=0$ this is true by definition, and we will now derive the estimate for $j>0$.

By definition of $\eta_j$ together with \cref{proposition:x-y-laplace-estimate-N_t} we have that
\begin{align}
\label{equation:X-norm-I-II-III-decomposition}
\begin{split}
\|{
\eta_j
}_{\mathfrak{X}_t}
&\leq
c
\|{
\Delta \eta_j
}_{\mathfrak{Y}_t}
\\
&\leq
c
\left(
\|{
\d^* \psi
}_{\mathfrak{Y}_t}
+
\|{
\d ^* (f_{j-1} \psi)
}_{\mathfrak{Y}_t}
+
\|{
*\d \left( F(\d \eta_{j-1}) \right)
}_{\mathfrak{Y}_t}
\right)
\\
&=
c \left(
I+II+III
\right).
\end{split}
\end{align}
By assumption we have
$I=\|{
\d^* \psi
}_{\mathfrak{Y}_t} \leq c' t^4$.

Now to estimate II:
\begin{align*}
\|{
\d ^* (f_{j-1} \psi)
}_{\mathfrak{Y}_t}
\leq
\|{
\d f_{j-1} \lrcorner \psi
}_{\mathfrak{Y}_t}
+
\|{
f_{j-1} \d^* \psi
}_{\mathfrak{Y}_t}
=
II.A+II.B.
\end{align*}
Here
\begin{align*}
II.A
&=
\|{\rho_t (\d f_{j-1} \lrcorner \psi)}_{C^{0,\alpha}_{\beta-2;t}}
+
t^{-3/2}
\|{\pi_t (\d f_{j-1} \lrcorner \psi)}_{C^{0,\alpha}}
\\
&\leq
(t^{-\alpha}+t^{-3/2-\alpha+\beta})
\|{\d f_{j-1} \lrcorner \psi}_{C^{0,\alpha}_{\beta-2;t}}
\\
&\leq
(t^{-\alpha}+t^{-3/2-\alpha+\beta})
\|{\d f_{j-1}}_{C^{0,\alpha}_{\beta-2;t}}
\|{\psi}_{C^{0,\alpha}_{0;t}}
\\
&\leq
c t^4,
\end{align*}
where for the first estimate we used \cref{proposition:iota-operator-norm-estimate,proposition:pi-operator-norm-estimate}, and for the last estimate we used the induction hypothesis $\|{\eta_{j-1}}_{\mathfrak{X}_t} \leq ct^4$, which implies $\|{\d f_{j-1}}_{C^{0,\alpha}_{\beta-2;t}} \leq ct^{7/2}$, together with the assumption $\|{\psi}_{C^{0,\alpha}_{0,0;t}} \leq c'' t^4$.
The estimate $II.B \leq ct^4$ is derived analogously.

It remains to estimate III:
\[
III
=
\|{\rho_t(*\d \left( F(\d \eta_{j-1}) \right)}
_{C^{0,\alpha}_{\beta-2;t}}
+
t^{-3/2}\|{\pi_t(*\d \left( F(\d \eta_{j-1}) \right)}
_{C^{0,\alpha}}
=
III.A+III.B.
\]
The summand III.A is estimated as
\begin{align*}
III.A
&\leq
ct^{-\alpha}
\|{
*\d \left( F(\d \eta_{j-1}) \right)
}_{C^{0,\alpha}_{\beta-2;t}},
\end{align*}
where we first estimate the $L^\infty$-part of the $C^{0,\alpha}$-norm.
Namely, by \cref{proposition:Theta-estimates}:
\begin{align*}
\|{
*\d \left( F(\d \eta_{j-1}) \right)
}_{L^\infty_{\beta-2;t}}
&\leq
c
\|{
\d \eta_{j-1}
}_{L^\infty_{\beta-1;t}}
\|{
\nabla \d \eta_{j-1}
}_{L^\infty_{\beta-2;t}}
t^{-1+\beta}
\\
&\quad
+
c
\|{
\d \eta_{j-1}
}_{L^\infty_{\beta-1;t}}^2
\|{
\d^* \psi
}_{L^\infty_{\beta-2;t}}
t^{-2+2\beta}
\\
&\leq
ct^4.
\end{align*}
The $[\cdot]_{C^{0,\alpha}}$-part is estimated analogously.
To estimate $
III.B=
t^{-3/2}
\|{
\pi_t
\left( *\d \left( F(\d \eta_{j-1}) \right) \right)
}_{C^{0,\alpha}}$, we again estimate the $L^\infty$-part first.
Fix some $y \in L$ and compute 
$\pi_t
\left( *\d \left( F(\d \eta_{j-1}) \right) \right)(y)$
by computing an integral over ${\XEH} \simeq \{y\} \times {\XEH} \subset L \times {\XEH}$.
By \cref{proposition:Theta-estimates} we have
\begin{align*}
\left|
\pi_t
\left( *\d \left( F(\d \eta_{j-1}) \right) \right)
\right|
&\leq
\left|
\<
*\d \left( F(\d \eta_{j-1}) \right),
\chi_t \nu
\>_{t^2g_{\XEH}}
\right|
\\
&\leq
c
\underbrace{
\int_{\XEH}
|\d \eta_{j-1}| \cdot |\nabla \d \eta_{j-1}|
\cdot
|\chi_t \nu|
\vol_{t^2g_{\XEH}}
}_{III.B.1}
\\
&\quad
+
c
\underbrace{
\int_{\XEH}
|\d \eta_{j-1}| \cdot |\d \eta_{j-1}| \cdot |\d^*\psi|
\cdot
|\chi_t \nu|
\vol_{t^2g_{\XEH}}
}_{III.B.2}.
\end{align*}
Here,
\begin{align*}
III.B.1
\cdot
t^{3/2}
&=
c
\int_{\XEH}
|\d \,(\overline{\pi}_t \eta_{j-1}+\rho_t \eta_{j-1})| 
\cdot 
|\nabla \d \, (\overline{\pi}_t \eta_{j-1}+\rho_t \eta_{j-1})|
\cdot
|\chi_t \nu|
\vol_{t^2g_{\XEH}}
\\
&\leq
c
\int_0^{\zeta}
\Big(
(t+r)^{-7}
\underbrace{
\|{\d \overline{\pi}_t \eta_{j-1}}_{C^{0,\alpha}_{-3;t}}
\|{\nabla \d \overline{\pi}_t \eta_{j-1}}_{C^{0,\alpha}_{-4;t}}
}_{
\leq
c
\|{\pi_t \eta_{j-1}}_{C^{2,\alpha}}^2
\leq
c
t^{2\cdot(4+3/2)}
}
\Big)
\left(
(t+r)^{-4}
t^2
\right)
r^3
\d r
\\
&\quad +
c
\int_0^{\zeta}
\Big(
(t+r)^{2\beta-3}
\underbrace{
\|{\d \rho_t \eta_{j-1}}_{C^{0,\alpha}_{\beta-1;t}}
\|{\nabla \d \rho_t \eta_{j-1}}_{C^{0,\alpha}_{\beta-2;t}}
}_{
\leq
c
\|{\rho_t \eta_{j-1}}_{C^{2,\alpha}_{\beta;t}}^2
\leq
c
t^{2 \cdot 4}
}
\Big)
\Big(
(t+r)^{-4}
t^2
\Big)
r^3
\d r
\\
&\quad
+
c
\int_0^{\zeta}
\Big(
(t+r)^{\beta-5}
\underbrace{
\|{\d \overline{\pi}_t \eta_{j-1}}_{C^{0,\alpha}_{-3;t}}
\|{\nabla \d \rho_t \eta_{j-1}}_{C^{0,\alpha}_{\beta-2;t}}
}_{
\leq
c
\|{\pi_t \eta_{j-1}}_{C^{2,\alpha}}
\|{\rho_t \eta_{j-1}}_{C^{2,\alpha}_{\beta;t}}
\leq
c
t^{4+3/2+4}
}
\Big)
\left(
(t+r)^{-4}
t^2
\right)
r^3
\d r
\\
&\quad
+
c
\int_0^{\zeta}
\Big(
(t+r)^{\beta-5}
\underbrace{
\|{\d \rho_t \eta_{j-1}}_{C^{0,\alpha}_{\beta-1;t}}
\|{\nabla \d \overline{\pi}_t \eta_{j-1}}_{C^{0,\alpha}_{-4;t}}
}_{
\leq
c
\|{\pi_t \eta_{j-1}}_{C^{2,\alpha}}
\|{\rho_t \eta_{j-1}}_{C^{2,\alpha}_{\beta;t}}
\leq
c
t^{4+3/2+4}
}
\Big)
\left(
(t+r)^{-4}
t^2
\right)
r^3
\d r
\\
&\leq
c
\left(
t^{2 \cdot(4+3/2)}t^{-7}t^{2}
+
t^{2 \cdot 4}t^{2\beta-3}t^{2}
+
2t^{4+3/2+4}t^{\beta-5}t^{2}
\right)
\\
&\leq
ct^{6},
\end{align*}
thus $III.B.1 \leq ct^4$.
The part $III.B.2$ and the $C^{0,\alpha}$-parts of $III.B.1$ and $III.B.2$ are estimated analogously.
Altogether, this gives $III \leq ct^4$ and therefore $\|{\eta_j}_{\mathfrak{X}_t}\leq ct^4$ by \cref{equation:X-norm-I-II-III-decomposition}.
It remains to show that the sequence $\eta_j$ has a limit, which will then turn out to be a solution to \cref{equation:joyce-slice-theorem-torsion-free-pde}.

The sequence $\eta_j$ satisfies
\begin{align*}
\|{\eta_j}_{C^{2,\alpha}_{\beta;t}}
&\leq
\|{\rho_t \eta_j}_{C^{2,\alpha}_{\beta;t}}
+
\|{\overline{\pi}_t \eta_j}_{C^{2,\alpha}_{\beta;t}}
\\
&\leq
\|{\eta_j}_{\mathfrak{X}_t}
+
t^{-2-\beta+3/2}
\|{\eta_j}_{\mathfrak{X}_t}
\\
&\leq
ct^{7/2-\beta}.
\end{align*}
As usual, the constant $c$ is independent of $t$, but in particular independent of $j$.
Thus, there exists, up to a subsequence, a $C^{2,\alpha/2}$-limit $\lim_{j \rightarrow \infty} \eta_j =: \eta$ by the Arzelà–Ascoli theorem.
This limit solves \cref{equation:joyce-slice-theorem-torsion-free-pde} and satisfies
\begin{align*}
\|{\eta}_{C^{2,\alpha/2}_{\beta;t}}
&\leq
ct^{7/2-\beta}.
\end{align*}
By \cref{proposition:iota-operator-norm-estimate,proposition:pi-operator-norm-estimate} we have that $\eta_j \rightarrow \eta$ also with respect to the norm $\|{\rho_t(\cdot)}_{C^{2,\alpha/2}_{\beta;t}}+t^{-3/2}\|{\pi_t (\cdot)}_{C^{2,\alpha/2}}$.
Thus, taking the limit $j \rightarrow \infty$ on both sides of
\begin{align*}
    \|{\rho_t(\eta_j)}_{C^{2,\alpha/2}_{\beta;t}}+t^{-3/2}\|{\pi_t (\eta_j)}_{C^{2,\alpha/2}}
    \leq
    c
    \|{\rho_t(\eta_j)}_{C^{2,\alpha}_{\beta;t}}+t^{-3/2}\|{\pi_t (\eta_j)}_{C^{2,\alpha}}
    =
    c\|{\eta_j}_{\mathfrak{X}_t}
    \leq
    ct^4
\end{align*}
yields the claimed $\|{\rho_t \eta}_{C^{2,\alpha/2}_{\beta;t}}+t^{-3/2}\|{\pi_t \eta}_{C^{2,\alpha/2}} \leq ct^4$.

By \cite[Theorem 10.3.7]{Joyce2000}, $\varphi+\d \eta$ is a torsion-free $G_2$-structure, which proves the claim.
\end{proof}

Taking everything together, this gives us:

\begin{theorem}
\label{corollary:kummer-construction-simplified-torsion-free-estimate}
Let $N_t$ be the resolution of $T^7/\Gamma$ from \cref{equation:resolution-of-t7-gamma} and $\varphi^t \in \Omega^3(N_t)$ the $G_2$-structure with small torsion from \cref{equation:g2-structure-on-kummer-construction}.
There exists $c>0$ independent of $t$ such that the following is true:
for $t$ small enough, there exists $\eta^t \in \Omega^2(N_t)$ such that $\tilde{\varphi}=\varphi^t+\d \eta^t$ is a torsion-free $G_2$-structure, and $\eta^t$ satisfies
\[
 \|{
  \eta^t
 }_{C^{2,\alpha/2}_{\beta;t}}
 \leq
 ct^{7/2-\beta}.
\]
In particular,
\[
\|{\tilde{\varphi}-\varphi^t}_{L^\infty} \leq ct^{5/2} \text{ and }
\|{\tilde{\varphi}-\varphi^t}_{C^{0,\alpha/2}} \leq ct^{5/2-\alpha/2} \text{ as well as }
\|{\tilde{\varphi}-\varphi^t}_{C^{1,\alpha/2}} \leq ct^{3/2-\alpha/2}.
\]
\end{theorem}

\begin{proof}
By \cref{lemma:kummer-construction-pregluing-estimate}, we have that $\|{ \psi}_{C^{0,\alpha}_{0;t}} \leq ct^4$.
Combined with \cref{proposition:pi-operator-norm-estimate,proposition:iota-operator-norm-estimate}, we also have
$\|{ \psi}_{\mathfrak{Y}_t} \leq ct^4$.
Thus, \cref{theorem:torsion-free-final-existence-theorem} can be applied, which gives the existence of $\eta^t \in \Omega^2(N_t)$ such that $\tilde{\varphi}=\varphi^t+\d \eta^t$ is a torsion-free $G_2$-structure and the estimate
\begin{align*}
 \|{
  \eta^t
 }_{C^{2,\alpha/2}_{\beta;t}}
 &\leq
 \|{\rho_t \eta^t}_{C^{2,\alpha/2}_{\beta;t}}
 +
 \|{\overline{\pi}_t \eta^t}_{C^{2,\alpha/2}_{\beta;t}}
 \\
 &\leq
 c
 \left(
 \|{\rho_t \eta^t}_{C^{2,\alpha/2}_{\beta;t}}
 +
 t^{-2-\beta}\|{\pi_t \eta^t}_{C^{2,\alpha/2}_{\beta;t}}
 \right)
 \\
 &\leq
 ct^{7/2-\beta},
\end{align*}
where we used $\rho_t+\overline{\pi}_t=1$ in the first step, we used \cref{proposition:iota-operator-norm-estimate} in the second step, and we used the estimate from \cref{theorem:torsion-free-final-existence-theorem} in the last step.
This implies the following estimate for the \emph{unweighted} $L^\infty$-norm:
\begin{align*}
    \|{\tilde{\varphi}-\varphi^t}_{L^\infty}
    \leq
    \|{\nabla \eta_t}_{L^\infty}
    \leq
    c\|{\nabla \eta_t}_{L^\infty_{\beta-1;t}} t^{\beta-1}
    \leq
    c t^{7/2-\beta} t^{\beta-1}
    =
    c t^{5/2}.
\end{align*}
The estimates for the unweighted Hölder norms follow analogously.
\end{proof}

\begin{remark}
 The power $7/2-\beta$ in \cref{corollary:kummer-construction-simplified-torsion-free-estimate} can be improved to $4-\epsilon$ for any $\epsilon \in (0,1)$ by defining the norms $\|{\cdot}_{\mathfrak{X}_t}$ and $\|{\cdot}_{\mathfrak{Y}_t}$ with a factor of $t^{-\kappa}$ instead of $t^{-3/2}$ for $\kappa \in (0,2)$ close to $2$.
\end{remark}

\begin{remark}
\label{remark:additional-g2-correction}
In \cite{Joyce1996a}, compact manifolds with holonomy $\Spin(7)$ were constructed.
 In the simplest case, one constructs $\Spin(7)$-structures with small torsion by gluing together the product $\Spin(7)$-structure on $T^4 \times {\XEH}$ to the flat $\Spin(7)$-structure on $T^8$.
 This gluing construction is analogous to the definition of the $G_2$-structure in \cref{equation:g2-structure-on-kummer-construction}.
 In contrast to the $G_2$-situation, however, Joyce's theorem about the existence of torsion-free $\Spin(7)$-structures cannot immediately be applied, because the torsion of the glued structure is too big.
 He overcame this problem by constructing a correction of the glued structure by hand which has smaller torsion, to which the existence theorem can be applied.
 The same can be done in the $G_2$ case.
 In fact, one gets a correction in the $G_2$-case from the $\Spin(7)$-case by considering the $\Spin(7)$-orbifold $T^7/\Gamma \times S^1$.
 Using this corrected structure, one would get even better control over the difference between glued structure and torsion-free structure than what is known from \cref{corollary:kummer-construction-simplified-torsion-free-estimate}.
\end{remark}

\section{Proof of the linear estimate \cref{proposition:x-y-laplace-estimate-N_t}}
\label{section:linear-estimate-proof}

This section covers the proof of \cref{proposition:x-y-laplace-estimate-N_t}, which is an estimate for the inverse of the Laplace operator on $N_t$.

One datum appearing in the linear estimate is the approximate kernel $\mathcal{K}$ defined in \cref{definition:approximate-kernel}.
It would be very easy to prove a linear estimate on $\mathcal{K}^\perp$ if one was allowed to take a very large $\mathcal{K}$.
Thus, it is important to check that our chosen $\mathcal{K}$ is not \emph{too large}.
More precisely, this means that the image of $\Delta$ does not become smaller when restricting to $\mathcal{K}^\perp$.
This check is carried out in \cref{subsection:approximate-kernel}.

The estimate is formulated in terms of the composite norms $\|{\cdot}_{\mathfrak{X}_t}$ and $\|{\cdot}_{\mathfrak{Y}_t}$.
In \cref{subsection:estimates-for-composite-norms} we prove some basic estimates for the auxiliary functions defining these norms.

To prove an estimate for the Laplacian on $N_t$, we combine two facts:
roughly speaking, we first prove that the Laplacian on $N_t$ on $2$-forms that are harmonic in the $\XEH$-direction can be identified with the Laplacian on functions on $L$, and we know that its kernel are exactly the locally constant functions on $L$.
This is done in \cref{subsection:comparison-with-laplacian-on-L}.
Second, prove that the Laplacian on $N_t$ satisfies an injectivity estimate, modulo $2$-forms that are harmonic in the $\XEH$-direction.
This is easy to prove, because we do not consider the very large space of $2$-forms that are harmonic in the $\XEH$-direction and therefore harder to analyse.
This is done in \cref{subsection:model-operator-on-R3-times-XEH}.

Combining both, we have an injectivity estimate for the Laplacian on $N_t$ on all $2$-forms:
those which are harmonic in the $\XEH$-direction as well as those which are not.
Actually concluding the proof in this fashion requires a small amount of extra work.
The words \emph{harmonic in the $\XEH$-direction} have a precise meaning on the space $L \times \XEH$, but on $N_t$ there are only approximate such forms, due to various cut-offs performed.
Thus, in order to prove the injectivity estimate of the Laplacian on $N_t$, we must estimate how these cut-offs interact with the Laplace operator on $N_t$ and $L$ respectively.
This is achieved in the fairly technical section \cref{subsection:cross-term-estimates}.

The proof of the injectivity estimate \cref{proposition:x-y-laplace-estimate-N_t} is then easily obtained by combining the previous estimates, which is done in the very short section \cref{subsection:proof-of-inectivity-estimate}.

\subsection{The approximate kernel}
\label{subsection:approximate-kernel}

The linear estimate \cref{proposition:x-y-laplace-estimate-N_t} only holds perpendicular to the approximate kernel defined in \cref{definition:approximate-kernel}.
The following proposition states that by restricting to the orthogonal complement of $\mathcal{K}$ we are not forgetting about any important $2$-forms --- the image of the Laplacian remains the same when restricted to this orthogonal complement.

\begin{proposition}
\label{proposition:approximate-kernel-not-too-big}
The operator
\[
\Delta:
\mathcal{K}^\perp
\rightarrow
\Im \Delta
\]
is surjective, where $\Im \Delta$ denotes the image of the Laplacian on all of $\Omega^2(N_t)$.
\end{proposition}

\begin{proof}
 \textbf{Step 1:}
 Show that the $L^2$-orthogonal projection $q: \Ker \Delta_{N_t} \rightarrow \mathcal{K}$ is an isomorphism.
 
 Assume there exists $0 \neq a \in \Omega^2(N_t)$ with $\Delta a=0$ such that $q(a)=0$, i.e. $a \perp \mathcal{K}$.
 Then $\Delta a \neq 0$ by \cref{proposition:x-y-laplace-estimate-N_t}, which is a contradiction.
 Now note $\dim (\Ker \Delta_{N_t})=b^0(L)+b^2(T^7/\Gamma)=12+k$, which is proved using the Künneth formula (see \cite[Proposition 6.1]{Joyce2017}).
 By construction, $\dim(\mathcal{K})=12+k$, so $q$ is a surjective linear map between vector spaces of the same dimension, and therefore injective.
 
 \textbf{Step 2:}
 Check 
 $\Im \left( \Delta|_{\mathcal{K}^\perp} \right)
 =
 \Im \Delta$.
 
 It suffices to check that $\Im \Delta \subset \Im \left( \Delta|_{\mathcal{K}^\perp} \right)$.
 Let $y \in \Im \Delta$, and $\Delta x=y$.
 Denote the $L^2$-orthogonal projection onto $\mathcal{K}$ by $\proj_{\mathcal{K}}$.
 Let
 \begin{align*}
  z
  :=
  q^{-1}(\proj_{\mathcal{K}}(-x)).
 \end{align*}
 Then $\Delta(x+z)=y$, and $\proj_{\mathcal{K}}(x+z)=0$ because of $\proj_{\mathcal{K}} \circ q^{-1}=\Id$, i.e. $x+z \perp \mathcal{K}$ which completes the proof.
\end{proof}

\subsection{Estimates for the composite norms}
\label{subsection:estimates-for-composite-norms}

In \cref{definition:x-y-norms} we defined the composite norms $\|{\cdot}_{\mathfrak{X}_t}$ and $\|{\cdot}_{\mathfrak{Y}_t}$.
These make use of two auxiliary functions:
roughly speaking, the map $\iota_t$ from \cref{equation:iota} that takes a function on $L$ and maps it to a $2$-form on $L \times \XEH$ which is harmonic in the $\XEH$-direction;
and the map $\pi_t$ from \cref{equation:def-pi} that is the converse.
In this subsection we prove some basic estimates for these two maps that will frequently be used throughout the rest of the section.

\begin{proposition}
\label{proposition:iota-operator-norm-estimate}
For all $k \in \N$ and $\beta > -4$ there exists $c > 0$ independent of $t$ such that for all $g \in \Omega^0(L)$ we have that
\begin{align}
\|{\iota_t g}_{C^{k,\alpha}_{\beta;t}}
\leq
c
t^{-2-\beta}
\|{g}_{C^{k,\alpha}}.
\end{align}
\end{proposition}

\begin{proof}
For the $L^\infty$-norm we have that
\begin{align*}
\|{
p_L^*g \cdot p_{\XEH}^*\nu
}_{L^\infty_{-4;t},g_{N_t}}
&\leq
c
\|{
p_L^*g \cdot p_{\XEH}^*\nu \cdot (t+t \check{r})^4
}_{L^\infty,g_{\R^3} \oplus t^2g_{\XEH}}
\\
&\leq
c
\|{
p_L^*g \cdot p_{\XEH}^*\nu \cdot (1+\check{r})^4 t^4 t^{-2}
}_{L^\infty,g_{\R^3} \oplus g_{\XEH}}
\\
&\leq
c t^2
\|{
p_L^*g
}_{L^\infty}
\end{align*}
where we used that $\nu = \mathcal{O}(\check{r}^{-4})$ and therefore
\begin{align}
\label{equation:nu-weighted-bound}
\|{\nu \cdot \check{r}^4}_{L^\infty,g_{\XEH}} \leq c,
\end{align}
in the last step.
For $\beta > -4$ we have that $\|{\chi_t}_{L^\infty_{4-\beta}} \leq ct^{-4-\beta}$, which proves the claim for the weighted $L^\infty$-norm.
The proof for higher derivatives is analogous.
\end{proof}

\begin{proposition}
\label{proposition:pi-operator-norm-estimate}
For all $k \in \N, \beta < 0$ there exists $c>0$ independent of $t$ such that for all $a \in \Omega^2(N_t)$ we have that
\begin{align}
\|{\pi_t a}_{C^{k,\alpha}}
\leq
t^{2+\beta-\alpha-k}
\|{a}_{C^{k,\alpha}_{\beta;t}}.
\end{align}
\end{proposition}

\begin{proof}
We first estimate the $L^\infty$-part, i.e. $\|{\pi_t a}_{L^\infty}$.
To this end
\begin{align*}
|\pi_t a(x)|
&\leq
\int_
{\{x \in \XEH: \check{r}(x) \leq t^{-1}\zeta\}}
|a|_{t^2 g_{\XEH}} \cdot |\nu |_{t^2 g_{\XEH}}
\vol _{t^2 g_{\XEH}}
\\
&\leq
t^2
\|{a}_{L^\infty_{\beta;t}}
\int_
{\XEH}
(t+\check{r}t)^\beta \cdot |\nu |_{g_{\XEH}}
\vol _{g_{\XEH}}
\\
&\leq
c
t^{2+\beta}
\|{a}_{L^\infty_{\beta;t}}
\int_
{\XEH}
(1+\check{r})^\beta \cdot (1+\check{r})^{-4}
\vol _{g_{\XEH}}
\\
&\leq
c
t^{2+\beta}
\|{a}_{L^\infty_{\beta;t}}
\underbrace{
\int_0^\infty
(1+\check{r})^{-4+\beta} \cdot \check{r}^3
\d \check{r}
}_{\leq c}
\\
&\leq
ct^{2+\beta}
\|{a}_{L^\infty_{\beta;t}},
\end{align*}
where in the second step we used the definition of $\|{\cdot }_{L^\infty_{\beta;t}}$ and switched from measuring in $t^2g_{\XEH}$ to measuring in $g_{\XEH}$ which introduces the factor of $t^2$;
in the third step we used
$|\nu |_{g_{\XEH}} 
\leq c(1+\check{r})^{-4}$;
in the fourth step we used polar coordinates to switch from integrating over ${\XEH}$ to integrating over $[0,\infty)$.
The estimates for the Hölder norm, derivatives, and for other weights are proved analogously.
\end{proof}

\subsection{Comparison with the Laplacian on $L$}
\label{subsection:comparison-with-laplacian-on-L}

The embedding $\iota_t: \Omega^0(L) \rightarrow \Omega^2(N_t)$ is defined in \cref{equation:iota} using a cut-off and rescaled version of $\nu \in \Omega^2({\XEH})$.
If not for this cut-off, we would have that $\Delta \iota_t=\iota_t \Delta$, where we use the symbol $\Delta$ to denote the Laplacian on $N_t$ as well as the Laplacian on $L$.
In our actual situation, we still have that $\Delta$ and $\iota_t$ nearly commute, and that is the content of the following proposition.

\begin{proposition}
\label{proposition:Delta-iota-commute}
For any $\beta \leq 0$ there exists $c>0$ independent of $t$ such that for all $g \in \Omega^0(L)$ we have
\begin{align}
\|{( \Delta \iota_t - \iota_t \Delta )g}_{C^{0,\alpha}_{\beta-2;t}}
\leq
c
t^2
\|{g}_{C^{2,\alpha}}.
\end{align}
\end{proposition}

\begin{proof}
Define the map $\tilde{\iota}_t:\Omega^0(L) \rightarrow \Omega^2(T^3 \times {\XEH})$ via $\tilde{\iota}_t(g)=p_L^* g \cdot p_{{\XEH}}^* \overline{\nu}$, where $\overline{\nu} \in \Omega^2({\XEH})$ is harmonic and has unit $L^2$-norm with respect to $g_{\XEH}$.
Then
\begin{align}
\label{equation:delta-iota-tilde-commute}
(\Delta \tilde{\iota}_t-\tilde{\iota}_t \Delta)g=0.
\end{align}
We aim to estimate
\begin{align*}
(\Delta \iota_t -\iota_t \Delta)g
&=
\underbrace{
(\Delta \iota_t - \Delta \tilde{\iota}_t)g
}_{=:I}
+
\underbrace{
(\Delta \tilde{\iota}_t-\tilde{\iota}_t \Delta)g
}_{=:II}
+
\underbrace{
(\tilde{\iota}_t \Delta-\iota_t \Delta)g
}_{=:III}.
\end{align*}
We begin by estimating I, where it will be convenient to estimate on two regions separately:
\begin{align}
\label{equation:omega-regions}
\begin{split}
\Omega_1
&:=
\{x \in L \times {\XEH}: \check{r}(x) \leq t^{-1}\zeta/8\},
\\
\Omega_2
&:=
\{x \in L \times {\XEH}: t^{-1}\zeta/8 \leq \check{r} (x) \leq t^{-1}\zeta/4\}.
\end{split}
\end{align}
Then
\begin{align*}
\|{I}_{C^{0,\alpha}_{\beta-2;t}}
&\leq
\|{
(\iota_t-\tilde{\iota}_t)g
}_{C^{2,\alpha}_{\beta;t}}
\\
&=
\|{
p_L^*g \cdot p_{\XEH}^*(\chi_t \nu'-\overline{\nu})
}_{C^{2,\alpha}_{\beta;t}}
\\
&\leq
\|{
p_L^*g \cdot p_{\XEH}^*(\chi_t \nu'-\overline{\nu})
}_{C^{2,\alpha}_{\beta;t}(\Omega_1)}
+
\|{
p_L^*g \cdot \chi_t p_{\XEH}^*(\chi_t \nu'-\overline{\nu})
}_{C^{2,\alpha}_{\beta;t}(\Omega_2)}.
\end{align*}

We will estimate the two summands separately.
The first summand is defined on the region $\Omega_1=\{x \in L \times {\XEH}: \check{r}(x) \leq t^{-1}\zeta/8\}$ where $\chi_t \equiv 1$, i.e. $\chi_t \cdot \nu'$ is not cut off.
The form $\overline{\nu}$ is nowhere cut off.
We have that 
\begin{align}
\label{equation:mu-nu-difference}
|\nu'(x)-\overline{\nu}(x)|_{t^2g_{\XEH}} \leq ct^2
\text{ for }
x \in {\XEH}
\text{ with }
\check{r}(x) \leq t^{-1}\zeta/8
\end{align} 
for the following reason:
it is $\< \overline{\nu}, \overline{\nu} \>_{L^2, t^2g_{\XEH}}=1$ by definition, thus
\begin{align*}
1
>
\<
\chi_t \overline{\nu}, \chi_t \overline{\nu} 
\>_{L^2,t^2g_{\XEH}}
&\geq
\<
\overline{\nu}, \overline{\nu}
\>_{L^2,t^2g_{\XEH}}
-
\int_{\{x \in {\XEH}: \check{r}(x) \geq \zeta t^{-1}/8 \}}
|\overline{\nu}|_{t^2g_{\XEH}}^2 \vol_{t^2g_{\XEH}}
\\
&
\geq
1
-
\int_{\zeta t^{-1}/8}^\infty
(1+r)^{-8}r^3 \d r
\geq
1-ct^4.
\end{align*}
If $\check{r}(x) \leq t^{-1}\zeta/8$ we have that
$\nu'(x)=\overline{\nu}(x) / 
\<
\chi_t \overline{\nu}, \chi_t \overline{\nu} 
\>_{L^2,t^2g_{\XEH}}$
because the cut-off of $\nu'$ is applied where $\check{r}(x) > t^{-1}\zeta/8$.
This implies, at the point $x$,
\begin{align*}
|\chi_t \nu'-\overline{\nu}|_{t^2g_{\XEH}} 
&\leq
\left|
\overline{\nu} \left( 1-\frac{1}{\< \chi_t \overline{\nu}, \chi_t \overline{\nu} \>_{L^2,t^2g_{\XEH}}} \right)
\right|_{t^2g_{\XEH}}
\\
&\leq
\left|
\overline{\nu} \cdot \frac{ct^4}{1-ct^4}
\right|_{t^2g_{\XEH}} 
\\
&\leq
t^{-2}
\left|
\overline{\nu} \cdot \frac{ct^4}{1-ct^4}
\right|_{g_{\XEH}}
\\
&\leq
ct^2.
\end{align*}
Using this for our estimate of the first summand of I, we obtain:
\begin{align*}
\|{
p_L^*g \cdot p_{\XEH}^*(\chi_t \nu'-\overline{\nu})
}_{C^{2,\alpha}_{\beta;t}(\Omega_1)}
&\leq
t^2 \|{p_L^*g}_{C^{2,\alpha}_{\beta;t}}
\leq
ct^2
\|{g}_{C^{2,\alpha}}
.
\end{align*}
For the second summand we get:
\begin{align*}
&\quad\|{
p_L^*g \cdot \chi_t p_{\XEH}^*(\chi_t \nu'-\overline{\nu})
}_{C^{2,\alpha}_{\beta;t}(\Omega_2)}
\\
&\leq
\|{p^*_{T^3}g}_{C^{2,\alpha}_{0;t}}
\|{
 \chi_t p_{\XEH}^*(\chi_t \nu'-\overline{\nu})
}_{C^{2,\alpha}_{\beta;t}(\Omega_2)}
\\
&\leq
\|{p^*_{T^3}g}_{C^{2,\alpha}_{0;t}}
\|{
1
}_{C^{2,\alpha}_{\beta+4;t}(\Omega_2)}
\left(
\|{\chi_t}_{C^{2,\alpha}_{0;t}}
\cdot
\|{
 \nu'
}_{C^{2,\alpha}_{-4;t}(\Omega_2)}
+
\|{
 \overline{\nu}
}_{C^{2,\alpha}_{-4;t}(\Omega_2)}
\right)
\\
&\leq
c t^2 \|{g}_{C^{2,\alpha}},
\end{align*}
where in the last step we used
$\|{
1
}_{C^{2,\alpha}_{\beta+4;t}(\Omega_2)} \leq c$,
which holds because far away from $L$, the weight function $w_{\beta+4;t}$ is uniformly bounded.
We also used
\begin{align}
\label{equation:nu-smallness-far-away-estimate}
|\overline{\nu} |_{t^2g_{\XEH}}=t^{-2} |\overline{\nu}|_{g_{\XEH}} \leq ct^{-2}(1+\check{r})^{-4} \leq ct^2 (t+t\check{r})^{-4} \leq ct^2 \text{ on } \Omega_2.
\end{align}
Together with \cref{equation:mu-nu-difference} this shows that $|\chi_t \nu' |_{t^2g_{\XEH}} \leq ct^2$ on $\Omega_2$.

Altogether $\|{I}_{C^{0,\alpha}_{\beta-2;t}} \leq ct^2 \|{g}_{C^{2,\alpha}}$.
Furthermore, $II=0$ because of \cref{equation:delta-iota-tilde-commute}.
Lastly, III is estimated like I, which shows the claim.
\end{proof}

The goal of this section is to prove \cref{proposition:x-y-laplace-estimate-N_t}, which is an estimate for the operator norm of the inverse of the Laplacian with respect to the norms $\|{ \cdot}_{\mathfrak{X}_t}$ and $\|{ \cdot}_{\mathfrak{Y}_t}$.
The purpose of these norms is to essentially split the problem into an estimate on $\Im \pi_t$ and remainder.
The following proposition contains the estimate on $\Im \pi_t$.

Such injectivity estimates can only hold perpendicular to the kernel of the linear operator.
On $L$, the kernel $\Ker \Delta_L$ of the Laplacian acting on functions are precisely the constant functions, and the condition that $g \perp \Ker \Delta_L$ is equivalent to $g$ having mean zero when integrated over $L$.

\begin{proposition}
\label{proposition:laplacian-perturbed-regularity}
There exists $c>0$ independent of $t$ such that for $t$ small enough and for all $g \in \Omega^0(L)$ satisfying $g \perp \Ker \Delta_L$ we have that
\begin{align}
\|{g}_{C^{2,\alpha}}
\leq
c
\|{ \pi_t \Delta \iota_t g}_{C^{0,\alpha}}.
\end{align}
\end{proposition}

\begin{proof}
We have
\begin{align*}
\|{g}_{C^{2,\alpha}}
&\leq
c
\|{\Delta g}_{C^{0,\alpha}}
\\
&=
c
\|{\pi_t \iota_t \Delta g}_{C^{0,\alpha}}
\\
&\leq
c
\|{\pi_t \Delta \iota_t g}_{C^{0,\alpha}}
+
c
\|{\pi_t \Delta \iota_t g-\pi_t \iota_t \Delta g}_{C^{0,\alpha}}
\\
&\leq
c
\|{\pi_t \Delta \iota_t g}_{C^{0,\alpha}}
+
ct^{2-\alpha} \|{g}_{C^{2,\alpha}},
\end{align*}
where we used elliptic regularity for the operator $\Delta$ on $L$ in the first step, and \cref{proposition:Delta-iota-commute,proposition:pi-operator-norm-estimate} in the last step.
At this point, the last summand $ct^{2-\alpha} \|{g}_{C^{2,\alpha}}$ can be absorbed into the left hand side for $t$ small enough.
\end{proof}

\subsection{An estimate on $N_t$ modulo fibrewise harmonic $2$-forms}
\label{subsection:model-operator-on-R3-times-XEH}

Recall the projection $\pi_t$ onto the fibrewise harmonic part from \cref{equation:def-pi} and its complement $\rho_t$.
In the preceding \cref{proposition:laplacian-perturbed-regularity} we essentially proved an estimate for the inverse of the Laplacian on $\Im \pi_t$.
In order to get an estimate with respect to $\|{\cdot}_{\mathfrak{X}_t}$ and $\|{\cdot}_{\mathfrak{Y}_t}$ we need to estimate the inverse of the Laplacian on $\Im \rho_t$.
To this end, we will prove an injectivity estimate for the Laplacian, with an extra term appearing on the right hand side that vanishes on $\Im \rho_t$.
This will be achieved in \cref{equation:rho-estimate-equation} of \cref{proposition:rho-estimate-on-N_t}, where the extra term on the right vanishes because $\pi_t \rho_t =0$.
We had previously defined $\overline{\pi}_t = \iota_t \pi_t$, and it turns out to be more convenient to work with this operator rather than $\pi_t$ directly, which is why it appears in \cref{proposition:rho-estimate-on-N_t}.

In the proof of \cref{proposition:rho-estimate-on-N_t} we will compare forms on $N_t$ with forms on the model space $\R^3 \times \XEH$, and we begin by making this comparison precise.

\begin{definition}
\label{definition:s_beta_t}
 For $j \in \{1, \dots, 12\}$ let $C'_j$ be a connected component of $\hat{U}$, but made slightly smaller, explicitly
 \[
  C'_j :=
  C_j \cap
  \{
   (x_h, x_v) \in L \times {\XEH}:
   \check{r}(x_v) \leq
   t^{-1} \zeta/4
  \}.
 \] 
 For $\beta \in \R$ let
 \begin{align*}
  s_{j,\beta,t}:
  \Omega^2(N_t) & \rightarrow \Omega^2(\R^3 \times \{x \in {\XEH}: \check{r}(x) \leq t^{-1} \zeta/4\})
  \\
  a
  &\mapsto
  t^{-\beta-2}
  ( p_t, \Id)^*
  \left( a|_{C'_j} \right),
 \end{align*}
 where $p_t: \R^3 \rightarrow T^3$, $p(x)=tx \mod \Z^3$ denotes a rescaled quotient map.
\end{definition}

Then:

\begin{lemma}
\label{lemma:s_t-second-lemma}
 For $j \in \{1, \dots, 12\}$, $\beta \in \R$ we have that for all $a \in \Omega^2(\R^3 \times {\XEH})$ we have
 \begin{align*}
  \|{
   s_{j,\beta,t}
   a
  }_{C^{k,\alpha}_{\beta}}
  &=
  \|{
   a
  }_{C^{k,\alpha}_{\beta;t}(C'_j)},
  \text{ and}
  \\
  \left(
  s_{j,\beta-2,t}
  \Delta_{N_t} a
  -
  \Delta_{g_{\R^3} \oplus g_{(1)}} s_{j,\beta,t} a
  \right)|_{C'_j}
  &=
  0.
 \end{align*}
 Here $\Delta_{g_{\R^3} \oplus g_{(1)}}$ denotes the Laplacian on $\R^3 \times {\XEH}$ with respect to the metric $g_{\R^3} \oplus g_{(1)}$.
\end{lemma}

\begin{proof}
 The map $(p_t, \Id): \R^3 \times \{x \in {\XEH}: \check{r}(x) \leq t^{-1} \zeta/4\} \rightarrow C'_j$ pulls back the metric induced by $\varphi^t$ defined in \cref{equation:g2-structure-on-kummer-construction} to the metric $t^2 (g_{\R^3} \oplus g_{(1)})$.
 (That is because rescaling the $\R^3$-direction introduces a factor $t^3$ in front of the summand $\delta_1 \wedge \delta_2 \wedge \delta_3$ from the definition of $\varphi^t$.)
 The extra factor $t^{-\beta-2}$ cancels out the factor $t^2$ when changing the metric from $t^2 (g_{\R^3} \oplus g_{(1)})$ to $g_{\R^3} \oplus g_{(1)}$ on $2$-forms and cancels out the factor $t^\beta$ coming from the definition of $\|{\cdot }_{C^{k,\alpha}_{\beta;t}}$.
\end{proof}

With this comparison between $N_t$ and the model space $\R^3 \times \XEH$ in place, we are ready to prove the following proposition, which is the anticipated injectivity estimate for the Laplacian modulo fibrewise harmonic $2$-forms, i.e. $2$-forms $a \in \Omega^2(N_t)$ in the image of $\rho_t$ from \cref{equation:pi-bar-definition}, which necessarily satisfy $\overline{\pi}_t(a)$ because of $\pi_t \rho_t=0$.
One also has to account for a finite-dimensional kernel of $\Delta$ on $N_t$ coming from harmonic forms on the orbifold $T^7/\Gamma$.
Because of this, we introduce a smaller approximate kernel $\mathcal{K}'$:

\begin{proposition}
\label{proposition:rho-estimate-on-N_t}
Write $\mathcal{K}':= \{ (1-\chi_t)u: u \in \Ker \Delta_{T^7/\Gamma} \} \subset \Omega^2(N_t)$.
Then there exists $c>0$ independent of $t$ such that for $a \in \Omega^2(N_t)$ satisfying $a \perp \mathcal{K}'$ we have
\begin{align}
\label{equation:rho-estimate-equation}
\|{a}_{C^{2,\alpha}_{\beta;t}}
\leq
c \left(
\|{\Delta a}_{C^{0,\alpha}_{\beta-2;t}}
+
\|{\overline{\pi}_t a}_{L^\infty_{\beta;t}}
\right).
\end{align}
\end{proposition}

\begin{proof}
 The Schauder estimate
 \begin{align}
 \label{equation:laplace-schauder-estimate}
  \|{
   a
  }_{C^{2,\alpha}_{\beta;t }}
  &\leq
  c'
  \left(
  \|{
   \Delta a
  }_{C^{0,\alpha}_{\beta-2;t}}
  +
  \|{
   a
  }_{L^\infty_{\beta;t }}
  \right)
 \end{align}
 can be derived as in \cite[Proposition 8.15]{Walpuski2017}.
 It then suffices to show:
\begin{align}
\label{equation:blowup-l-infty-estimate}
\text{there exists $c$ such that }
\|{a}_{L^\infty_{\beta;t}}
\leq
c \left(
\|{\Delta a}_{C^{0,\alpha}_{\beta-2;t}}
+
\|{\overline{\pi}_t a}_{L^\infty_{\beta;t}}
\right)
\text{ for all }
a \perp \mathcal{K}'.
\end{align}
 Assume \cref{equation:blowup-l-infty-estimate} is false, then there exist $t_i \rightarrow 0$, $a_i \in \Omega^2(N_{t_i})$ satisfying $a_i \perp \mathcal{K}'$, and $x_i \in N_{t_i}$ such that
 \begin{align}
  \|{a}_{L^\infty_{\beta;t_i}} \leq c,
  \quad
  \left|
   w_{t_i}^\beta(x_i) a_i(x_i)
  \right|
  =1,
  \text{ and }
  \|{
   \Delta a_i
  }_{C^{0,\alpha}_{\beta-2;t_i}}
  \rightarrow 0,
  \quad
  \|{
   \overline{\pi}_{t_i} a_i
  }_{L^\infty_{\beta;t_i}}
  \rightarrow 0
    .
 \end{align}
 Together with \cref{equation:laplace-schauder-estimate} this implies: $\|{
   a
  }_{C^{2,\alpha}_{\beta;t_i }}
  \leq
  c$
 from \cref{equation:laplace-schauder-estimate}.
 Without loss of generality we can assume to be in one of three following cases, and we will arrive at a contradiction in each of them.
 
 \textbf{Case 1:}
 the sequence $x_i$ concentrates on one ALE space, i.e. $t_i^{-1} r_{t_i}(x_i) \rightarrow R < \infty$ (see \cref{figure:blow-up-1}).
 
 \begin{figure}[htbp]
  \begin{center}
   \includegraphics[width=10cm]{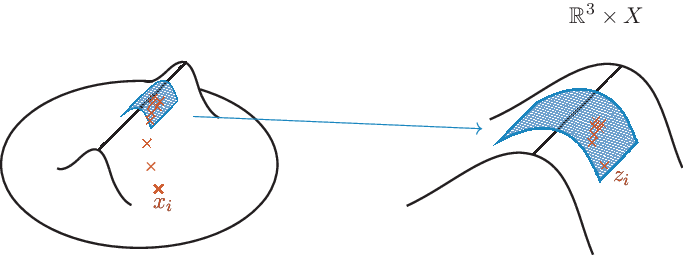}
  \end{center}
  \caption{Blowup analysis near the associative is reduced to the analysis of the Laplacian on $\R^3 \times {\XEH}$. The figure is taken from \cite{Platt2024}.}
  \label{figure:blow-up-1}
 \end{figure}
 
 By passing to a subsequence and translating in the $\R^3$-direction if necessary, we can assume that $x_i$ concentrates near one fixed connected component of $L$.
 Let $C_j \subset L \times {\XEH}$ be the connected component $\hat{U}$ containing an accumulation point of the sequence $x_i$.
 Define $\tilde{a}_i := s_{j,\beta,t} a_i \in \Omega^2(\R^3 \times \{ x \in {\XEH}: \check{r}(x) \leq t_i^{-1} \zeta/4 \})$ and let $\tilde{x}_i$ be a lift from $C_j$ to $\R^3 \times {\XEH}$.
 The new $2$-form $\tilde{a}_i$ then satisfies
 \begin{align*}
  \|{
   \tilde{a}_i
  }_{C^{2,\alpha}_{\beta}}
  \leq
  c,
  (1+\check{r}(\tilde{x}_i))^{-\beta}
  \left|
   \tilde{a}_i(\tilde{x}_i)
  \right|
  \geq
  c,
  \text{ and }
  \|{ \Delta \tilde{a}_i}_{C^{0,\alpha}_{\beta-2}} \rightarrow 0,
 \end{align*}
 which follows from \cref{lemma:s_t-second-lemma}.
 Now the weight function no longer has $t_i$ in it and distances and tensors are measured using the metric $g_{\R^3} \oplus g_{(1)}$.
 
 By the assumption of case 1, we have $\check{r}(\tilde{x}_i) \rightarrow R < \infty$.
 By passing to a subsequence we can assume that $\tilde{x}_i$ converges, so write $x^* := \lim_{i \rightarrow \infty} \tilde{x}_i \in \R^3 \times {\XEH}$.
 Using the Arzelà-Ascoli theorem and a diagonal argument, we can extract a limit $a^* \in \Omega^2(\R^3 \times {\XEH})$ of the sequence $\tilde{a}_i$ satisfying:
 \begin{align}
 \label{equation:norm-of-a*}
  \|{
   a^*
  }_{L^\infty_{\beta}}
  \leq
  c,
  \text{ and }
  \\
  \label{equation:a*-in-kernel-of-laplacian}
  \Delta_{g_{\R^3} \oplus g_{(1)}} \, a^* =0,
  \text{ and }
  \\
  \label{equation:a*-nonzero}
  (1+\check{r}(x^*))^{-\beta}
  \left|
   a^*(x^*)
  \right|
  \geq c.
 \end{align}
 By \cref{lemma:laplacian-independent-of-r-3-direction} (applied to the case $\R^3 \times {\XEH}$), we have that $a^*$ is independent of the $\R^3$-direction.
 By \cref{proposition:hoelder-kernel-of-laplacian}, the only harmonic forms on ${\XEH}$ that decay like $\check{r}^\beta$ are multiples of $\nu$.
 Thus $a^*$ is the pullback of a multiple of $\nu$ under the projection $p_{\XEH}: \R^3 \times {\XEH} \rightarrow {\XEH}$.
 
 Because $\|{
   \overline{\pi}_{t_i} a_i
  }_{L^\infty_{\beta;t_i}}
  \rightarrow 0$, we have that $a^*$ is perpendicular to $\overline{\nu}$ on every $\{y \} \times {\XEH} \subset \R^3 \times {\XEH}$.
  Here is how to see this in detail:
  let $y \in L$, then we calculate on $\{y\} \times {\XEH}$:
  \begin{align}
  \label{equation:limit-orthogonal}
   \<
    a^*, \overline{\nu}
   \>
   =
   \<
    a^*, \overline{\nu}-\chi_t \nu 
   \>
   +
   \<
    a^*-\tilde{a}_i, \chi_t \nu 
   \>
   +
   \<
    \tilde{a}_i, \chi_t \nu 
   \>
   =
   I+II+III.
  \end{align}
  Here,
  \begin{align*}
   | I |
   &\leq
   \left|
   \<
    a^*, \overline{\nu}-\chi_t \nu
   \>_{\{x \in {\XEH}: \check{r}(x) \leq t^{-1}\zeta/8 \}}
   \right|
   +
   \left|
   \<
    a^*, \overline{\nu}-\chi_t \nu
   \>_{\{x \in {\XEH}: \check{r}(x) \geq t^{-1}\zeta/8 \}}
   \right|,
  \end{align*}
  where we have for the first summand
  \begin{align*}
   \left|
   \<
    a^*, \overline{\nu}-\chi_t \nu
   \>_{\{x \in {\XEH}: \check{r}(x) \leq t^{-1}\zeta/8 \}}
   \right|
   &\leq
   \int_0^{t^{-1} \zeta/8}
   |a^*|_{g_{(1)}}
   \cdot
   |\overline{\nu}-\chi_t \nu|_{g_{(1)}} r^3 \d r
   \\
   &\leq
   c
   \int_0^{t^{-1} \zeta/8}
   r^\beta t^4 r^3 \d r
   \leq
   ct^{-\beta}
   \rightarrow 0.
  \end{align*}
  Here we used \cref{equation:norm-of-a*} and \cref{equation:mu-nu-difference} (after changing from $|\cdot |_{t^2 g_{\XEH}}$ to $|\cdot |_{g_{\XEH}}$) in the second step.
  For the second summand we find
  \begin{align*}
   \left|
   \<
    a^*, \overline{\nu}-\chi_t \nu
   \>_{\{x \in {\XEH}: \check{r}(x) \geq t^{-1}\zeta/8 \}}
   \right|
   &\leq
   c
   \int_{\zeta/8 t^{-1}}^\infty
   r^\beta r^{-4} r^3 \d r
   \leq
   ct^{-\beta} \rightarrow 0,
  \end{align*}
  where we used $\overline{\nu}=\mathcal{O}(\check{r}^{-4})$ and $\nu=\mathcal{O}(\check{r}^{-4})$ in the first step.
  
  In order to estimate $II$, let $l>0$.
  Then
  \begin{align*}
   |II|
   &\leq
   \left|
   \<
    a^*-\tilde{a}_i, \chi_t \nu
   \>_{\{x \in {\XEH}: \check{r}(x) \geq l \}}
   \right|
   +
   \left|
   \<
    a^*-\tilde{a}_i, \chi_t \nu
   \>_{\{x \in {\XEH}: \check{r}(x) \leq l \}}
   \right|,   
  \end{align*}
  and we find for the first summand
  \begin{align*}
   \left|
   \<
    a^*-\tilde{a}_i, \chi_t \nu
   \>_{\{x \in {\XEH}: \check{r}(x) \geq l \}}
   \right|
   &\leq
   c
   \left(
   \|{a^*}_{L^\infty_\beta}
   +
   \|{\tilde{a}_i}_{L^\infty_\beta}
   \right)
   \int_l^\infty
   r^{\beta-4+3} \d r
   \leq
   cl^\beta  
  \end{align*}
  for a constant $c$ independent of $l$.
  For the second summand we have
  \begin{align*}
   \left|
   \<
    a^*-\tilde{a}_i, \chi_t \nu
   \>_{\{x \in {\XEH}: \check{r}(x) \leq l \}}
   \right|
   &\leq
   \|{ a^*-\tilde{a}_i}_{L^\infty_\beta(\{ x \in {\XEH}: \check{r}(x) \leq l\}) }
   \cdot
   \int_0^l
   r^{\beta-4+3} \d r
   \\
   &\leq
   c 
   \|{ a^*-\tilde{a}_i}_{L^\infty_\beta(\{ x \in {\XEH}: \check{r}(x) \leq l\}) }
   \rightarrow
   0  
  \end{align*}
  as $i \rightarrow \infty$ by definition of $a^*$.
  Last,
  \begin{align*}
   | III |
   &=
   t^{-2-\beta}
   | (\pi_t a_i)(y) |
   =
   t^{-2-\beta}
   | (\pi_t \iota_t \pi_t a_i)(y) |   
   \leq
   c
   \|{ \overline{\pi}_t a_i }_{L^\infty_{\beta;t}}
   \rightarrow 0,
  \end{align*}
  where we used \cref{proposition:pi-operator-norm-estimate} for the estimate.
  
  Altogether we see that, by taking $\lim _{i \rightarrow \infty}$ in \cref{equation:limit-orthogonal}, we have that
  $\< a^*, \overline{\nu} \> \leq cl^\beta$, where the constant $c$ was independent of $l$.
  This is true for any $l > 0$, therefore $\< a^*, \overline{\nu} \> =0$.
  But this is a contradiction to \cref{equation:a*-nonzero}.
  
 \textbf{Case 2:}
 the sequence $x_i$ concentrates on the regular part, i.e. $r_{t_i}(x_i) \rightarrow R > 0$ (see \cref{figure:blow-up-2}).
 
 \begin{figure}[htbp]
  \begin{center}
   \includegraphics[width=10cm]{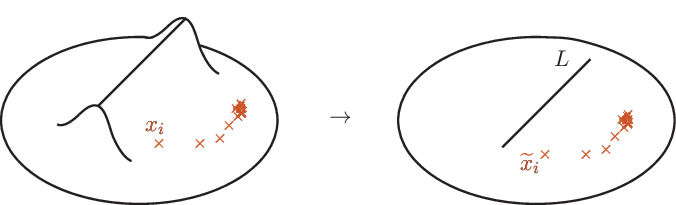}
  \end{center}
  \caption{Blowup analysis away from the associative is reduced to the analysis of the Laplacian on $T^7/\Gamma$. The figure is taken from \cite{Platt2024}.}
  \label{figure:blow-up-2}
 \end{figure}
 
 Using the Arzelà-Ascoli theorem and a diagonal argument, we can extract a limit $a^* \in \Omega^2(T^7 / \Gamma \setminus L)$.
 Denote, furthermore, $\lim _{i \rightarrow \infty} x_i = x^*$.
 We have $\left| a^* \right| < c \cdot d(\cdot , L)^{\beta}$, so we have that $a^*$ is a well-defined distribution on $M/\< \iota \>$ acting on $L^2$-sections because $\beta > -2$.
 We also have $\Delta a^*=0$, so $a^*$ is smooth by elliptic regularity, e.g. \cite[Theorem 6.33]{Folland1995}.
 
 Furthermore,
 \begin{align}
 \label{equation:orthogonality-on-regular-part-concentration}
  \<
   a^*,
   (1-\chi(2d(\cdot, L))) \cdot \alpha_i
  \>_{T^7/\Gamma}
  &=
  \lim_{i \rightarrow \infty}
  \<
   a_i,
   (1-\chi_t(r_t)) \cdot \pi^* \alpha_i
  \>_{N_{t_i}}
  =0.
 \end{align}
 By the unique continuation property for elliptic PDEs, the inner product
 \begin{align*}\< \; \cdot \;, (1-\chi) \circ (2d (\cdot, L)) \; \cdot \; \>\end{align*} is non-degenerate on harmonic forms.
 The $2$-form $a^*$ is a harmonic form that is orthogonal to all harmonic forms with respect to this inner product, therefore $a^*=0$.
 But this contradicts $a^*(x^*)>c$.
 
 \textbf{Case 3:}
 the sequence $x_i$ concentrates on the neck region, i.e. $\check{r}(x_i) \rightarrow \infty$, but $r_t(x_i) \rightarrow 0$ (see \cref{figure:blow-up-3}).
 
 \begin{figure}[htbp]
  \begin{center}
   \includegraphics[width=10cm]{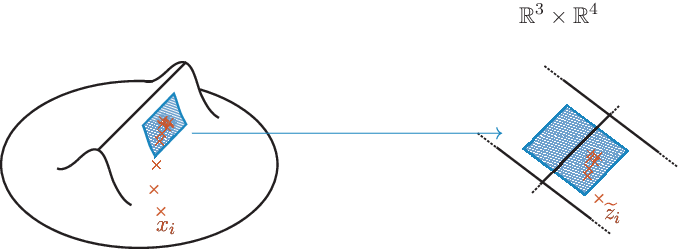}
  \end{center}
  \caption{Blowup analysis in the neck region is reduced to the analysis of the Laplacian on $\R^3 \times \R^4$. The figure is taken from \cite{Platt2024}.}
  \label{figure:blow-up-3}
 \end{figure}
 
 Define $\tilde{a}_i \in \Omega^2(\R^3 \times {\XEH})$ and $\tilde{x}_i \in \R^3 \times {\XEH}$ as in case 1.
 In this case, we have that $\left| \rho (\tilde{x}_i) \right| \rightarrow \infty$.
 In order to be able to obtain a limit of this sequence, let $R_i \rightarrow \infty$ be a sequence such that $R_i/\left| \rho (\tilde{x}_i) \right| \rightarrow 0$.
 Cutting out the exceptional locus of the Eguchi-Hanson space, we can consider $\{ (x_h,x_v) \in \R^3 \times {\XEH} : R_i \leq \left| \rho \right|(x_v) \leq \zeta t_i^{-1} \}$ as a subset of $\R^3 \times \C^2/\{\pm 1\}$.
 On $\R^3 \times \C^2/\{\pm 1\}$, we have the rescaling map $(\cdot \left| \rho (\tilde{x}_i) \right|)$.
 
 We now define:
 \begin{align}
 \begin{split}
  \tilde{\tilde{a}}_i
  &:=
  (\cdot \left| \rho (\tilde{x}_i) \right|)^*
  \left(
   \tilde{a}_i |_{\{ R_i \leq \left| \rho \right| \leq \zeta t_i^{-1} \}}
  \right)
  \cdot
  \left| \rho (\tilde{x}_i) \right|^{-2-\beta}
  \\
  &
  \;\;\;\;\;\; \in
  \Omega^2(\R^3 \times \{ x \in {\XEH}: R_i/\left| \rho (\tilde{x}_i) \right| \leq \left| \rho(x) \right| \leq \zeta t_i^{-1}/\left| \rho (\tilde{x}_i) \right| \}),
  \\
  \tilde{\tilde{x}}_i
  &:=
  \tilde{x}_i/\left| \rho (\tilde{x}_i) \right|.
 \end{split}
 \end{align}
 This sequence satisfies
 \begin{align}
 \label{equation:neck-limit-properties}
  \|{
   \tilde{\tilde{a}}_i
  }_{C^{2,\alpha}_{\beta}}
  \leq
  c
  \text{ and }
  \left|
   \tilde{\tilde{a}}_i(\tilde{\tilde{x}}_i)
  \right|
  >
  c.
 \end{align}
 The data $\tilde{\tilde{a}}_i$ and $\tilde{\tilde{x}}_i$ are defined on (subsets of) $\R^3 \times \C^2/\{ \pm 1 \}$.
 We use the same symbols to denote their pullbacks under the quotient map $\C^2 \rightarrow \C^2/\{ \pm 1 \}$.
 
 As before, we extract a $C^{2,\alpha/2}_{loc}$-limit $a^* \in \Omega^2(\R^3 \times \R^4 \setminus \{0 \})$ satisfying
 \begin{align*}
  \Delta_{\R^7} a^* =0,
  \text{ and }
  \|{
   a^*
  }_{L^\infty_{\beta}(\R^3 \times \R^4)}
  \leq c.
 \end{align*}
 We see as in case 2 that $a^*$ defines a distribution on all of $\R^7$, and is smooth by elliptic regularity on all of $\R^7$.
 
 We also get an $L^\infty$-bound for $a^*$ as follows:
 away from $\R^3 \times \{0\}$, this is given by \cref{equation:neck-limit-properties}.
 To see that $a^*$ does not blow up in the $\R^3$-direction near $\R^3 \times \{0\}$, consider any $y \in \R^3 \times \{0\}$.
 Let $1 < p < -4/\beta$, then $\|{a^*}_{L^p(B_1(y))} \leq c$, independent of $y$, by \cref{equation:neck-limit-properties}.
 So, by elliptic regularity $\|{a^*}_{L_m^p(B_1(y))} \leq c$ for any $m \in \mathbb{N}$, and by the Sobolev embedding we have $\|{a^*}_{L^\infty} \leq c$, where all of these estimates were independent of $y$.
 
 By \cref{lemma:laplacian-independent-of-r-3-direction} (applied to $\R^3 \times \R^4$), $a^*$ is constant in the $\R^3$ direction.
 The limit $a^*$ is therefore the pullback of a harmonic, bounded form of $\R^4$, and must thus vanish, which is a contradiction to the second part of \cref{equation:neck-limit-properties}.
\end{proof}

\subsection{Cross-term estimates}
\label{subsection:cross-term-estimates}
We have now established uniform estimates for the inverse of $\Delta$ on $\Im \pi_t$ and $\Im \rho_t$.
As it stands, it could happen that the operator norm of $\rho_t \Delta \overline{\pi}_t$ or $\pi_t \Delta \rho_t$ is very big.
It will turn out in our proof of \cref{proposition:x-y-laplace-estimate-N_t} that in such a case one would be unable to deduce anything about the inverse of the operator norm of $\Delta$ with respect to $\|{ \cdot}_{\mathfrak{X}_t}$ and $\|{ \cdot}_{\mathfrak{Y}_t}$.
Fortunately, it turns out that the operator norms of $\rho_t \Delta \iota_t$ (and therefore $\rho_t \Delta \overline{\pi}_t$, because $\overline{\pi}_t=\iota_t \pi_t$) and $\pi_t \Delta \rho_t$ are small, which is the content of the following proposition.

\begin{proposition}
\label{proposition:laplace-cross-term-estimates}
There exists $c>0$ independent of $t$ such that for all $g \in \Omega^0(L)$ and for all $a \in \Omega^2(N_t)$ we have
\begin{align}
\label{equation:rho-cross-term-estimate}
\|{\rho_t \Delta \iota_t g}_{C^{0,\alpha}_{\beta;t}}
&\leq
c
t^{2-\alpha}
\|{g}_{C^{2,\alpha}}
\text{ if }
\beta<0,
\\
\label{equation:improved-cross-term-estimate}
\|{\pi_t \Delta \rho_t a}_{C^{0,\alpha}}
&
\leq
c
t^{2+2\beta-2\alpha}
\|{\rho_t a}_{C^{2,\alpha}_{\beta;t}}
\text{ if }
-2<\beta<0.
\end{align}
\end{proposition}

\begin{proof}
We first prove \cref{equation:rho-cross-term-estimate}.
We have $\rho_t \iota_t=0$ and therefore
\begin{align*}
\|{\rho_t \Delta \iota_t g}_{C^{0,\alpha}_{\beta;t}}
&=
\|{\rho_t(\Delta \iota_t g-\iota_t \Delta g)}_{C^{0,\alpha}_{\beta;t}}
\\
&\leq
\|{\Delta \iota_t g-\iota_t \Delta g}_{C^{0,\alpha}_{\beta;t}}
+
\|{\iota_t \pi_t(\Delta \iota_t g-\iota_t \Delta g)}_{C^{0,\alpha}_{\beta;t}}
\\
&\leq
\|{\Delta \iota_t g-\iota_t \Delta g}_{C^{0,\alpha}_{\beta;t}}
+
ct^{-2-\beta}
\|{\pi_t(\Delta \iota_t g-\iota_t \Delta g)}_{C^{0,\alpha}}
\\
&\leq
\|{\Delta \iota_t g-\iota_t \Delta g}_{C^{0,\alpha}_{\beta;t}}
+
ct^{-\alpha}
\|{\Delta \iota_t g-\iota_t \Delta g}_{C^{0,\alpha}_{\beta;t}}
\\
&\leq
ct^{2-\alpha}
\|{g}_{C^{2,\alpha}},
\end{align*}
where we used \cref{proposition:iota-operator-norm-estimate} in the third step,
\cref{proposition:pi-operator-norm-estimate} in the fourth step,
and \cref{proposition:Delta-iota-commute} in the last step.

Now to prove \cref{equation:improved-cross-term-estimate}:
assume without loss of generality that $a=\rho_t a$.
Define
\begin{align*}
\tilde{\pi}_t : \Omega^2(T^3 \times {\XEH}) &\rightarrow \Omega^0(L)
\\
(\tilde{\pi}_t a)(x)
&:=
\< a, \overline{\nu} \>_{t^2g_{\XEH}}.
\end{align*}
The difference between $\tilde{\pi}_t$ and $\pi_t$ is that they use $\overline{\nu}$ and $\chi_t\nu$ in their definition, respectively:
$\overline{\nu}$ is not cut off, $\chi_t\nu$ is, and both are rescaled to have unit norm.
It suffices to prove the claim for $a \in \Omega^2(N_t)$ which is supported near $L$.
We can view such $a$ as an element in $\Omega^2(T^3 \times {\XEH})$ and apply $\tilde{\pi}_t$ to it.
Also define $\tilde{\iota}_t:\Omega^0(L) \rightarrow \Omega^2(T^3 \times {\XEH})$ as $\tilde{\iota}_t(g)=p^*_{T^3} \cdot p^*_{\XEH} \overline{\nu}$.
Then $\tilde{\pi}_t\tilde{\iota}_t=\Id$ and we also define $\tilde{\rho}_t:=1-\tilde{\iota}_t \tilde{\pi}_t$.

We have $\tilde{\pi}_t \Delta=\Delta \tilde{\pi}_t$, thus $\tilde{\pi}_t a =0 \Rightarrow \tilde{\pi}_t \Delta a=0$, and therefore $\tilde{\pi}_t \Delta \tilde{\rho}_t=0$.
Hence
\begin{align*}
\pi_t \Delta \rho_t a
&=
\underbrace{
(\pi_t -\tilde{\pi}_t) \Delta \rho_t a
}_{=:I}
+
\underbrace{
\tilde{\pi}_t \Delta (\rho_t-(1-\iota_t \tilde{\pi}_t))a
}_{=:II}
+
\underbrace{
\tilde{\pi}_t \Delta ((1-\iota_t\tilde{\pi}_t)-\tilde{\rho}_t) a
}_{=:III}
.
\end{align*}

We first estimate I:
\begin{align*}
\< \Delta \rho_t a, \overline{\nu} -\chi_t\nu \>_{L^2,t^2g_{\XEH}}
&\leq
\underbrace{
c
t^{4+\beta}
\int_0^{t^{-1}\zeta/8}
\left(
\|{\Delta \rho_t a}_{C^{0,\alpha}_{\beta-2;t}} (1+r)^{-2+\beta}
\right)
r^3 \d r
}_{\leq 
ct^{2+\beta}\|{\rho_t a}_{C^{2,\alpha}_{\beta;t}} \text{ if }
-2 \leq \beta \leq 0
}
\\
&\quad
+
\underbrace{
ct^\beta
\int_{t^{-1}\zeta/8}^\infty
\|{\rho_t a}_{C^{2,\alpha}_{\beta;t}}
(1+r)^{-2+\beta-4} r^3 \d r
}_{\leq ct^2\|{\rho_t a}_{C^{2,\alpha}_{\beta;t}}}.
\end{align*}
Here we applied \cref{equation:mu-nu-difference} on the region $\{x \in {\XEH}:\check{r}(x) \leq \zeta t^{-1}/8\}$ 
and we used
\[
|\overline{\nu} -\chi_t\nu|_{t^2g_{\XEH}} 
\leq 
|\overline{\nu}|_{t^2g_{\XEH}}+|\chi_t\nu|_{t^2g_{\XEH}} 
\leq 
c(t+ \check{r} t)^{-4}t^2
\]
on the region $\{x \in {\XEH}: \check{r}(x) \geq \zeta t^{-1}/8\}$.
Thus
\begin{align*}
\|{
(\pi_t -\tilde{\pi}_t) \Delta \rho_t a
}_{L^\infty}
\leq
ct^{2+\beta}
\|{\rho_t a}_{C^{2,\alpha}_{\beta;t}}
\end{align*}
and the $C^{0,\alpha}$-estimate follows analogously.

For estimating II we need the estimate
\begin{align}
\label{equation:pi-tilde-operator-norm-estimate}
\|{\tilde{\pi}_t a}_{C^{k,\alpha}}
\leq
t^{2+\beta-\alpha-k}
\|{a}_{C^{k,\alpha}_{\beta;t}}.
\end{align}
which is proved like \cref{proposition:pi-operator-norm-estimate}.
Then
\begin{align*}
\|{
\tilde{\pi}_t \Delta (\rho_t-(1-\iota_t \tilde{\pi}_t))a
}_{C^{0,\alpha}}
&=
\|{
\tilde{\pi}_t \Delta (\iota_t\pi_t-\iota_t \tilde{\pi}_t)a
}_{C^{0,\alpha}}
\\
&\leq
ct^{-\alpha}
\|{
\Delta \iota_t (\pi_t-\tilde{\pi}_t)a
}_{C^{0,\alpha}_{-2;t}}
\\
&\leq
ct^{-\alpha}
\left(
\|{\iota_t \Delta (\pi_t - \tilde{\pi}_t)a}
_{C^{0,\alpha}_{-2;t}}
+
t^2
\|{
(\pi_t-\tilde{\pi}_t)a
}_{C^{2,\alpha}}
\right)
\\
&\leq
ct^{-\alpha}
(1+t^2)
\|{
(\pi_t-\tilde{\pi}_t)a
}_{C^{2,\alpha}}
\\
&\leq
ct^{-\alpha}
(1+t^2)
t^2
\|{
a
}_{C^{2,\alpha}_{\beta;t}}
\\
&\leq
ct^{2-\alpha}
\|{
\rho_t a
}_{C^{2,\alpha}_{\beta;t}}
\end{align*}
where in the first estimate we used \cref{equation:pi-tilde-operator-norm-estimate},
in the second estimate we used \cref{proposition:Delta-iota-commute},
in the third estimate we used the estimate for the operator norm of $\iota_t$ from \cref{proposition:iota-operator-norm-estimate},
and in the fourth estimate we did the same calculation as when estimating I and we again used $-2< \beta <0$.
In the last step we used the assumption that $a=\rho_t a$.

It remains to estimate III.
We find
\begin{align*}
\|{
\tilde{\pi}_t \Delta ((1-\iota_t\tilde{\pi}_t)-\tilde{\rho}_t) a
}_{C^{0,\alpha}}
&=
\|{
\tilde{\pi}_t \Delta (\iota_t-\tilde{\iota}_t)\tilde{\pi}_t a
}_{C^{0,\alpha}}
\\
&\leq
ct^{-\alpha+\beta}
\|{
\Delta (\iota_t-\tilde{\iota}_t)\tilde{\pi}_t a
}_{C^{0,\alpha}_{\beta-2;t}}
\\
&\leq
ct^{-\alpha+\beta}
\|{
(\iota_t-\tilde{\iota}_t)\Delta \tilde{\pi}_t a
}_{C^{0,\alpha}_{\beta-2;t}}
+
t^{2-\alpha+\beta}
\|{
\tilde{\pi}_t a
}_{C^{2,\alpha}},
\end{align*}
where we used \cref{equation:pi-tilde-operator-norm-estimate} in the second step, and $\tilde{\iota}_t \Delta=\Delta \tilde{\iota}_t$ together with \cref{proposition:Delta-iota-commute} in the third step.
Here we find for the first summand
\begin{align*}
ct^{-\alpha+\beta}
\|{
(\iota_t-\tilde{\iota}_t)\Delta \tilde{\pi}_t a
}_{C^{0,\alpha}_{\beta-2;t}}
&\leq
ct^{-\alpha+\beta}
\|{
\chi_t\nu-\overline{\nu}
}_{C^{0,\alpha}_{0;t}}
\cdot
\|{
p_L^* \Delta \tilde{\pi}_t a
}_{C^{0,\alpha}_{0;t}}
\cdot
\|{
1
}_{C^{0,\alpha}_{\beta-2;t}}
\\
&\leq
ct^{-\alpha+\beta}
\cdot
t^2
\cdot
\|{
\Delta \tilde{\pi}_t a
}_{C^{0,\alpha}}
\\
&\leq
ct^{2-\alpha+\beta}
\cdot
\|{
\tilde{\pi}_t a
}_{C^{2,\alpha}}
\\
&\leq
ct^{2-2\alpha+2\beta}
\cdot
\|{
a
}_{C^{2,\alpha}_{\beta;t}}
\end{align*}
where we used \cref{equation:mu-nu-difference,equation:nu-smallness-far-away-estimate} in the second step;
we used
$\|{
p_L^* \Delta \tilde{\pi}_t a
}_{C^{0,\alpha}_{0;t}}=
\|{
\Delta \tilde{\pi}_t a
}_{C^{0,\alpha}}$ which holds because $p_L^* \Delta \tilde{\pi}_t a$ is constant in the Eguchi-Hanson direction, so the derivative in the $C^{0,\alpha}_{0;t}$-norm is just a derivative in the $L$-direction;
in the last step we used \cref{equation:pi-tilde-operator-norm-estimate}.
For the second summand we have
\begin{align*}
t^{2-\alpha+\beta}
\|{
\tilde{\pi}_t a
}_{C^{2,\alpha}}
&\leq
t^{2-2\alpha+2\beta}
\|{
a
}_{C^{2,\alpha}_{\beta;t}}
\end{align*}
by \cref{equation:pi-tilde-operator-norm-estimate}, which proves the claim.
\end{proof}

\subsection{Proof of \cref{proposition:x-y-laplace-estimate-N_t}}
\label{subsection:proof-of-inectivity-estimate}

\begin{proof}[Proof of \cref{proposition:x-y-laplace-estimate-N_t}]
By definition, 
$\|{a}_{\mathfrak{X}_t}=
\|{\rho_t a}_{C^{2,\alpha}_{\beta;t}}
+
t^{-3/2}
\|{\pi_t a}_{C^{2,\alpha}}$.
We treat the first summand first:
\begin{align*}
\|{\rho_t a}_{C^{2,\alpha}_{\beta;t}}
&\leq
\|{\Delta \rho_t a}_{C^{0,\alpha}_{\beta-2;t}}
\\
&\leq
\left(
\|{ \overline{\pi}_t \Delta \rho_t a}_{C^{0,\alpha}_{\beta-2;t}}
+
\|{ \rho_t \Delta a}_{C^{0,\alpha}_{\beta-2;t}}
+
\|{ \rho_t \Delta \overline{\pi}_t a}_{C^{0,\alpha}_{\beta-2;t}}
\right),
\end{align*}
where we used \cref{proposition:rho-estimate-on-N_t} in the first step and in the second step used $1=\overline{\pi}_t+\rho_t$ twice.
Here, the first summand satisfies
\begin{align*}
\|{ \overline{\pi}_t \Delta \rho_t a}_{C^{0,\alpha}_{\beta-2;t}}
&
\leq
t^{-\beta}
\|{ \pi_t \Delta \rho_t a}_{C^{0,\alpha}}
\\
&
\leq
t^{\beta+2-2\alpha}
\|{ \rho_t a}_{C^{2,\alpha}_{\beta;t}},
\end{align*}
where we used \cref{proposition:iota-operator-norm-estimate} in the first step,
and \cref{equation:improved-cross-term-estimate} in the second step.
The resulting term can be absorbed into the left hand side of \cref{equation:x-y-laplace-estimate}.

For the third summand we get from \cref{equation:rho-cross-term-estimate} that
\begin{align*}
\|{ \rho_t \Delta \overline{\pi}_t a}_{C^{0,\alpha}_{\beta-2;t}}
&
\leq
ct^{2-\alpha}
\|{ \pi_t a}_{C^{2,\alpha}},
\end{align*}
which can be absorbed into the left hand side of \cref{equation:x-y-laplace-estimate} if $\alpha$ is sufficiently small.

Regarding the $\pi_t$-term, we find that
\begin{align*}
t^{-3/2}
\|{\pi_t a}_{C^{2,\alpha}}
&\leq
t^{-3/2}
\|{\pi_t \Delta \iota_t \pi_t a}_{C^{0,\alpha}}
\\
&\leq
t^{-3/2}
\left(
\|{\pi_t \Delta a}_{C^{0,\alpha}}
+
\|{\pi_t \Delta \rho_t a}_{C^{0,\alpha}}
\right),
\end{align*}
where we used \cref{proposition:laplacian-perturbed-regularity} in the first step and $1=\overline{\pi}_t+\rho_t$ in the second step.
Here we have for the last summand
\begin{align}
\label{equation:alpha-beta-smallness-condition}
t^{-3/2}
\|{\pi_t \Delta \rho_t a}_{C^{0,\alpha}}
&\leq
t^{-3/2}
t^{2+2\beta-2\alpha}
\|{\rho_t a}_{C^{2,\alpha}_{\beta;t}}
\end{align}
which can be absorbed into the left hand side of \cref{equation:x-y-laplace-estimate}.
The remaining terms, i.e. the ones that have not been absorbed into the left hand side of \cref{equation:x-y-laplace-estimate}, exactly sum up to $\|{\Delta a}_{\mathfrak{Y}_t}$, which proves the claim.
\end{proof}

\cleardoublepage

\addcontentsline{toc}{section}{References}

\bibliographystyle{alpha}
\bibliography{library}

\end{document}